\documentclass[11pt]{article}
\usepackage{amsmath}
\usepackage{amssymb,amsthm,latexsym}
\usepackage{xypic,graphicx}

\theoremstyle{plain}
 \newtheorem{theorem}{Theorem}[subsection]
 \newtheorem{lemma}[theorem]{Lemma}
 \newtheorem{proposition}[theorem]{Proposition}
 \newtheorem{corollary}[theorem]{Corollary}
  \newtheorem{question}[theorem]{Question}

 \theoremstyle{definition}
 \newtheorem{definition}[theorem]{Definition}

 \theoremstyle{remark}
 
 \newtheorem{remark}[theorem]{Remark}

\DeclareMathOperator{\pol}{pol}

\DeclareMathOperator{\spec}{spec}
\DeclareMathOperator{\TR}{TR}
\DeclareMathOperator{\Tr}{Tr}
\DeclareMathOperator{\tr}{tr}

\DeclareMathOperator{\ind}{index}
\DeclareMathOperator{\End}{End}

\DeclareMathOperator{\sgn}{sgn}
\DeclareMathOperator{\Spin}{Spin}
\DeclareMathOperator{\supp}{supp}
\DeclareMathOperator{\Ad}{Ad}
\DeclareMathOperator{\vol}{vol}

\DeclareMathOperator{\AS}{AS}
\DeclareMathOperator{\Aps}{APS}

\DeclareMathOperator{\LIM}{LIM}
\DeclareMathOperator{\reg}{reg}

\newcommand{\Spinc}{\Spin^c}

\newcommand{\C}{\ensuremath{\mathbb{C}}}

\newcommand{\N}{\ensuremath{\mathbb{N}}}

\newcommand{\R}{\ensuremath{\mathbb{R}}}

\newcommand{\Z}{\ensuremath{\mathbb{Z}}}

\newcommand{\Ca}[1]{\ensuremath{\mathcal{#1}}}
\newcommand{\cA}{\Ca{A}}
\newcommand{\cB}{\Ca{B}}

\newcommand{\cN}{\Ca{N}}

\newcommand{\cC}{\Ca{C}}
\newcommand{\cS}{\ensuremath{\mathcal{S}}}

\newcommand{\APS}{APS}

\newcommand{\Zg}{Z}

\newcommand{\kg}{\mathfrak{g}}
\newcommand{\kk}{\mathfrak{k}}
\newcommand{\kp}{\mathfrak{p}}
\newcommand{\ka}{\mathfrak{a}}

\newcommand{\ba}{\begin{eqnarray}}
   \newcommand{\na}{\end{eqnarray}}

\newcommand{\beq}[1]{\begin{equation} \label{#1}}
\newcommand{\eeq}{\end{equation}}

\numberwithin{equation}{section}

\begin{document}

 \title{An equivariant Atiyah--Patodi--Singer index theorem for proper actions I: the index formula}

\author{Peter Hochs\footnote{Radboud University, 
\texttt{p.hochs@math.ru.nl}},
 Bai-Ling Wang\footnote{Australian National University, 
\texttt{bai-ling.wang@anu.edu.au} } {} and Hang Wang\footnote{East China Normal University, 
\texttt{wanghang@math.ecnu.edu.cn}}
}

\maketitle

\begin{abstract}
Consider a proper, isometric action by a unimodular locally compact group $G$ on a Riemannian manifold $M$ with boundary, such that $M/G$ is compact. For an equivariant, elliptic operator $D$ on $M$, and an element $g \in G$, we define a numerical index $\ind_g(D)$, in terms of a parametrix for $D$ and a trace associated to $g$. We prove an equivariant Atiyah--Patodi--Singer index theorem for this index. We first state general analytic conditions under which this theorem holds, and then show that these conditions are satisfied if $g=e$ is the identity element; if  $G$ is a finitely generated, discrete group, and the conjugacy class of $g$ has polynomial growth; and  if $G$ is a connected, linear, real semisimple Lie group, and $g$ is a semisimple element. 
   In the classical case, where $M$ is compact and $G$ is trivial, our arguments reduce to a relatively short and simple proof of the original Atiyah--Patodi--Singer index theorem.
   
In part II of this series, we prove that, under certain conditions, $\ind_g(D)$ can be recovered from a $K$-theoretic index of $D$ via a trace defined by the orbital integral over the conjugacy class of $g$.
\end{abstract}

\setcounter{tocdepth}{1}

\tableofcontents

\section{Introduction}

\subsection{Background}

For a Dirac-type operator $D$ on a compact, even-dimensional manifold $M$ with boundary $N$, the Atiyah--Patodi--Singer (\APS) index theorem \cite{APS1} states that
\beq{eq APS intro}
\ind_{\Aps}(D) = \int_M \AS(D) - \frac{\dim \ker(D_N)+\eta(D_N)}{2}.
\eeq
Here $\ind_{\Aps}(D)$ is the analytic index of $D$ as an operator between spaces of sections satisfying the {\APS} boundary conditions; $\AS(D)$ is the Atiyah--Singer integrand corresponding to $D$; $D_N$ is a Dirac operator on the boundary $N$, compatible with $D$, and $\eta(D_N)$ is its \emph{$\eta$-invariant}.

This $\eta$-invariant is the key additional ingredient compared to the Atiyah--Singer index theorem for the case of closed manifolds \cite{ASI}. It measures the spectral asymmetry of $D_N$. If $\{\lambda_j\}_{j=1}^{\infty}$ are the nonzero eigenvalues of $D_N$ with multiplicities, then one considers the series
\[
\sum_{j=1}^{\infty} \frac{\sgn(\lambda_j)}{|\lambda_j|^z},
\]
for $z$ with large enough real part so that the series converges. The meromorphic continuation of this expression is regular at $z=0$, and $\eta(D_N)$ is its value at $z=0$. An equivalent expression is
\[
\eta(D_N) = \frac{2}{\sqrt{\pi}}\int_{0}^{\infty} \Tr(D_N e^{-t^2D_N^2})\, dt,
\]
where $\Tr$ is the operator trace.
The integral converges in many cases, for example for twisted $\Spin$-Dirac operators \cite{Bismut86}. Otherwise, one replaces the lower integration limit by $\varepsilon>0$, and defines the integral as the constant term in the asymptotic expansion in powers of $\varepsilon$, as $\varepsilon \downarrow 0$.

The equality \eqref{eq APS intro} has been generalised in several directions. Space does not permit us to give a complete overview of the literature here, but we briefly mention some of the earlier work that is most relevant to the present paper, where we study equivariant indices on noncompact manifolds.

Donnelly \cite{Donnelly} proved an equivariant version of \eqref{eq APS intro}, where a compact group acts on $M$, generalising the Atiyah--Segal--Singer fixed point theorem \cite{Atiyah68, ASIII} to manifolds with boundary. This was generalised to manifolds that do not have a product form near the boundary by Braverman--Maschler \cite{Braverman17a}.

On noncompact manifolds, {\APS} index theory of Callias-type Dirac operators was developed by Braverman and Shi \cite{Braverman17c, Braverman17b, Shi17, Shi18}. 

After this paper appeared, Piazza, Posthuma, Song and Tang generalised our results in the case of semisimple Lie groups \cite{PPST}. Their definition of the index involves $b$-calculus and is different from ours, but likely yields the same number. 

Various authors refined \eqref{eq APS intro} by considering the action by the fundamental group $G = \pi_1(M)$ on the universal cover $\overline{M}$ of $M$. A first version of the $\eta$-invariant that applies in this setting is Cheeger--Gromov's $L^2$-$\eta$-invariant \cite{Cheeger85}. This corresponds to the identity element of $G$ in the sense of~\cite{Lott92}. This type of $\eta$-invariant was used by Ramachandran \cite{Ramachandran} to generalise \eqref{eq APS intro} to a statement about the lift $\overline{D}$ of $D$ to $\overline{M}$.

To include information corresponding to nontrivial conjugacy classes of $G$, Lott developed the notions of \emph{higher} and \emph{delocalised $\eta$-invariants} \cite{Lott92, Lott99}. (See also the $\hat \eta$-form of \cite{Bismut89}.) The higher $\eta$-invariant lies in a homology space of an algebra of rapidly decaying functions on $G$. These invariants are interesting for their own sakes, and were also used to state and prove generalisations of \eqref{eq APS intro}. Generalisations of Ramachandran's result involving this higher $\eta$-invariant were proved in the context of $b$-calculus by Leichtnam--Piazza \cite{Leichtnam97, Leichtnam98, Leichtnam00}. These constructions and results involve assumptions on $G$ and its elements used, such as polynomial growth of conjugacy classes.

Xie--Yu \cite{XieYu} worked in the context of coarse geometry, and showed that Lott's delocalised $\eta$-invariant can be recovered from a $K$-theoretic $\rho$-invariant. They generalised \eqref{eq APS intro} to a version for nontrivial elements of $G$ whose conjugacy classes have polynomial growth, for the action by $G$ on $\overline{M}$. Since this action is free, and the first term on the right hand side of \eqref{eq APS intro} becomes an integral over the fixed point set of the group element in question, that term vanishes in this context. A version for higher cyclic cocycles was proved by Chen--Wang--Xie--Yu, see Theorem 7.3 in \cite{CWXY19}.

In this paper, we work in the more general setting of proper actions by locally compact groups, instead of just fundamental groups acting freely on universal covers. We obtain an equivariant version of \eqref{eq APS intro} in this case.




\subsection{The main results}

Consider a unimodular, locally compact group $G$ acting properly and isometrically on a Riemannian manifold $M$, with boundary $N$, such that $M/G$ is compact. Let $D$ be a $G$-equivariant Dirac-type operator on a $G$-equivariant, $\Z_2$-graded Hermitian vector bundle $E  = E_+ \oplus E_- \to M$. Suppose that all structures have a product form near $N$. In particular, suppose that near $N$, the restriction of $D$ to sections of $E_+$ equals
\[
\sigma \Bigl(-\frac{\partial}{\partial u} + D_N\Bigr),
\]
where $\sigma\colon E_+|_N \to E_-|_N$ is an equivariant vector bundle isomorphism, $u$ is the coordinate in $(0,1]$ in a neighbourhood of $N$ equivariantly isometric to $N \times(0,1]$, and
$D_N$ is a Dirac operator on $E_+|_N$.

Let $g \in G$, and let $\Zg$ be its centraliser in $G$. Suppose that $G/\Zg$ has a $G$-invariant measure $d(h\Zg)$.
We will define an index of $D$ in terms of  the \emph{$g$-trace} of a $G$-equivariant operator $T$ on $L^2(E)$, defined as follows. Let $\chi$ be a nonnegative function on $M$ whose support has compact intersections with all $G$-orbits, such that for all $m \in M$,
\beq{eq chi intro}
\int_G \chi(gm)^2\, dg = 1.
\eeq
A $G$-equivariant operator $T$ on $L^2(E_+)$ or $L^2(E_-)$ is said to be \emph{$g$-trace class} if  it has a smooth kernel $\kappa$, and the integral
\[
\int_{G/{\Zg}} \int_M \chi(hgh^{-1}m)^2 \tr(hgh^{-1} \kappa(hg^{-1}h^{-1} m, m))\, dm\, d(h\Zg)
\]
converges absolutely. This integral is then its $g$-trace, denoted by $\Tr_g(T)$.
An equivariant, possibly unbounded operator $F$ from $L^2(E_+)$ to $L^2(E_-)$ is \emph{$g$-Fredholm} if it has a parametrix $R$ such that the operators $1-RF$ and $1-FR$ are $g$-trace class. The \emph{$g$-index} of $F$, denoted by $\ind_g(F)$ is then the difference of the $g$-traces of these operators. For odd, self-adjoint operators on $L^2(E)$, we define the $g$-index as the $g$-index of their restrictions to even-graded sections.

The $g$-trace and $g$-index can be defined analogously for operators on $N$. 
The \emph{$g$-delocalised $\eta$-invariant}of $D_N$ is
\beq{eq deloc eta intro}
\eta_g(D_N) := \frac{2}{\sqrt{\pi}}\int_0^{\infty} \Tr_g(D_N e^{-t^2 D_N^2})\, dt,
\eeq
whenever the integrand is defined for all $t>0$ and the integral converges. (The term \emph{delocalised} is often reserved for the case $g\not= e$. We will use it for $g=e$ as well.)

Let $\tilde M$ be the double of $M$,
let $\tilde E \to \tilde M$ be the extension of $E$,
and let $\tilde D$ be the extension of $D$ to $\tilde M$. (To be precise, on the second copy of $M$ inside $\tilde M$, both the orientation on $M$ and the grading on $E$ are reversed.) We write $\tilde D_{\pm}$ for the restrictions of $\tilde D$ to sections of $\tilde E_{\pm}$, respectively.
We state the results here under the condition that $D_N$ is invertible for simplicity; the results generalise to the case where $0$ is isolated in the spectrum of $D_N$.
\begin{theorem}\label{thm APS intro}
Suppose 
 that
 $D_N$ is invertible and has a well-defined $g$-delocalised $\eta$-invariant. Suppose that 
 the conditions listed in Theorem \ref{thm APS} hold. Then the operators $D$ and $\tilde D$ are $g$-Fredholm, and
\beq{eq APS intro 2}
\ind_g(D) = \frac{1}{2}\ind_g(\tilde D)- \frac12{\eta_g(D_N)}.
\eeq
\end{theorem}
See Theorem \ref{thm APS}. The manifold $\tilde M$ has no boundary, so that the first term on the right hand side of \eqref{eq APS intro 2} is the contribution of the interior of $M$.


We make Theorem \ref{thm APS intro} more concrete if $g$ is the identity element $e$ of an arbitrary unimodular locally compact group, and for more general elements $g$ of two relevant classes of groups. We do this by showing that the conditions of Theorem \ref{thm APS intro} hold, and giving a topological expression for the first term on the right hand side of \eqref{eq APS intro 2}. This has the following consequence.
\begin{corollary} \label{cor APS intro}
Suppose that $D$ is a $\Spinc$-Dirac operator twisted by a vector bundle, and that $D_N$ is invertible. Suppose furthermore that 
either
\begin{itemize}
\item $G$ is any unimodular locally compact group, and $g=e$;
\item $G$ is discrete and finitely generated, and the conjugacy class of $g$ has polynomial growth; or
\item $G$ is a connected, linear, real semisimple Lie group, and $g$ is semisimple.
\end{itemize}
Then $D$ is $g$-Fredholm, 
 $D_N$ has a well-defined $g$-delocalised $\eta$-invariant, and
\[
\ind_g(D)=\int_{M^g} \chi_g^2 \AS_g( D)- \frac{1}{2} \eta_g(D_N),
\]
where $\AS_g( D)$ is the integrand in the Atiyah--Segal--Singer fixed point formula, and $\chi_g$ is defined analogously to \eqref{eq chi intro}, for the action by ${\Zg}$ on the fixed-point set $M^g$.
\end{corollary}
 See Corollaries \ref{cor e}, \ref{cor APS discr} and \ref{cor APS ss}.  The topological expressions for the contributions from $\tilde M$, i.e.\ the contributions of the interior of $M$, are the main results in \cite{HW2, Wangwang, Wang14}.
 
 Theorem \ref{thm APS intro} and Corollary \ref{cor APS intro} generalise to the case where $D_N$ is not necessarily invertible, but $0$ is isolated in its spectrum. See Theorem \ref{thm DN not inv} and Corollary \ref{cor DN not inv}.
These results 
simultaneously generalises the main results in \cite{APS1, Donnelly} to the noncompact case, the main result in \cite{Ramachandran} and Theorem 5.3 in \cite{XieYu} to the case of general proper actions,  by a larger class of groups, and the main results in \cite{HW2, Wangwang, Wang14} to the case of manifolds with boundary.

In part II of this series \cite{HWWII}, we prove that, under certain conditions, $\ind_g(D)$ can be recovered from a $K$-theoretic index of $D$ via a trace defined by the orbital integral over the conjugacy class of $g$. Then Theorem \ref{thm APS intro} and Corollary \ref{cor APS intro} can be used to study that index as well.

\subsection{Features of the proofs and future work}

Our approach to proving an equivariant {\APS} index theorem for proper actions is to work with general locally compact groups as long as possible, and to find general conditions under which such an index theorem holds. This leads to Theorem \ref{thm APS intro}. We then check these conditions $g=e$ and for discrete and semisimple groups, and obtain the more concrete Corollary \ref{cor APS intro}. We then discuss how to weaken the assumption that $D_N$ is invertible to the assumption that zero is isolated in its spectrum, if it is in the spectrum.

As expected, the techniques used to show that the conditions of Theorem \ref{thm APS intro} hold are very different for discrete groups and semisimple Lie groups.  Nevertheless, one of our aims here is to develop a unified approach to this type of index theorem as far as possible. For other classes of groups, such index theorems can now be obtained by checking the conditions of Theorem \ref{thm APS intro}. And the fact that the explicit index formula in Corollary \ref{cor APS intro} holds for  two quite different classes of groups suggests that  this formula should hold more generally.

Our proofs of Theorem \ref{thm APS intro} and Corollary \ref{cor APS intro} follow a different approach from  proofs of earlier results. 
 In the technical heart of our proof, we apply the $g$-trace to a sum of operators that individually are not trace-class for this trace. By adding and subtracting another operator in a suitable way, we are able to split up that trace into two computable terms. `Regularising' traces in this way is related to the $b$-calculus approach to the {\APS} index theorem, as Richard Melrose pointed out to us. In part II \cite{HWWII}, we link the approach in this paper to an equivariant index defined in terms of coarse geometry and Roe algebras.
 We have not found the combination of coarse index theory and such techniques related to $b$-calculus elsewhere in the literature, and this is a potentially novel feature of our approach.

The notion of $g$-trace class operators that we use is a weak one, and does reduce to the usual notion of trace class operators if $G$ is trivial, for example. We use this weak notion, because the operators $1-RF$ and $1-FR$, as below \eqref{eq chi intro}, that are relevant in our setting, have this property, and not (obviously) a stronger property. This is enough for the proof of Theorem \ref{thm APS intro}. A price we pay for using this generalised notion of $g$-trace class operators is that compositions, or even squares, of $g$-trace class operators are not obviously $g$-trace class. This is one of the technical points that is solved in part II \cite{HWWII}, to obtain a link with a $K$-theoretic index.

In the classical case, where $M$ is compact and $G$ is trivial (or compact), many technical parts of our proofs are unnecessary.
 Then only simplified versions of the arguments in Subsections \ref{sec interior}--\ref{sec pf APS}, and parts of Section \ref{sec DN not inv} are needed to prove \eqref{eq APS intro}. This leads to a proof of the original {\APS} index theorem that, to the authors, seems considerably shorter and simpler than the proof in \cite{APS1}. As pointed out in the previous paragraph, this proof involves a regularisation procedure for traces of initially non-trace class operators, analogous to the $b$-trace in $b$-calculus \cite{Melrose}.


Part of our motivation for proving Corollary \ref{cor APS intro} is potential future applications to Borel--Serre compactifications of noncompact locally symmetric spaces. 
These spaces are finite-volume orbifolds of the form $\Gamma \backslash G /K$, where $G$ is a semisimple Lie group, $K<G$ is maximal compact, and $\Gamma < G$ is a cofinite-volume lattice. If $\Gamma \backslash G /K$ is noncompact, a natural compactification is its Borel--Serre compactification, which in general is an orbifold with corners. By either generalising Corollary \ref{cor APS intro} to manifolds with corners, or, initially, considering cases where the Borel--Serre compactification is a manifold with boundary, for example, $G=SL(2, \R), K=SO(2)$ and $\Gamma=SL(2, \Z)$, we intend to investigate applications of the results in this paper to equivariant index theory for the action by $\Gamma$ on the space whose quotient is the Borel--Serre compactification of  $\Gamma \backslash G /K$, and associated trace formulas. 
Such applications would be equivariant refinements of \cite{ADS83, ADS84}.
This may require an extension of Corollary \ref{cor APS intro} to conjugacy classes that do not have polynomial growth; see Remark \ref{rem non poly growth}. More generally, the case of Corollary \ref{cor APS intro} for discrete groups has  potential applications to orbifolds with boundary.

\subsection{Outline of this paper}

In Section \ref{sec prelim}, we introduce the $g$-trace and the $g$-index, Lott's delocalised $\eta$-invariants, and state our main results: Theorem \ref{thm APS} and Corollaries \ref{cor e}, \ref{cor APS discr} and \ref{cor APS ss}. In Section \ref{sec prop trace ind}, we discuss some properties of the $g$-trace and the $g$-index. We use these properties in Section \ref{sec conv eta} to prove convergence of delocalised $\eta$-invariants in several cases, and then to prove the main results in Section \ref{sec proof}. Up to this point, we assume that the Dirac operator $D_N$ on the boundary is invertible. We discuss how to remove that assumption in Section \ref{sec DN not inv}. In Appendix \ref{sec cpt}, we show how the proof of the index theorem simplifies in the original compact case.

\subsection*{Acknowledgements}

We thank Richard Melrose, Yanli Song, Xiang Tang and the referees for helpful comments. PH thanks East China Normal University for funding a visit there in 2018.

This work was supported by the Australian Research Council, through a Discovery Early Career Researcher Award
[grant number DE160100525], by the Shanghai Rising-Star Program  [grant number 19QA1403200], and by the National Natural Science Foundation of China [grant number 11801178].

\section{Preliminaries and results} \label{sec prelim}

\subsection{The $g$-trace and the $g$-index} \label{sec orbital}

Let $G$ be a locally compact group, with a left Haar measure $dg$.
Let $g \in G$, and let ${\Zg}$ be its centraliser. Suppose that $G/{\Zg}$ has a $G$-invariant measure $d(h{\Zg})$.
Let $M$ be a complete Riemannian manifold, on which $G$ acts properly and isometrically.
Let $\chi \in C^{\infty}(M)$ be nonnegative, and such that for all $m \in M$,
\[
\int_G \chi(gm)^2\, dg = 1.
\]
Let $E \to M$ be a $G$-equivariant, Hermitian vector bundle. We will write
\[
\End(E) := E \boxtimes E^* \to M \times M.
\]
\begin{definition} \label{def g trace class}
A section $\kappa \in \Gamma^{\infty}(\End(E))^G$ is \emph{$g$-trace class} if the integral
\beq{eq def Trg}
\int_{G/{\Zg}} \int_M \chi(hgh^{-1}m)^2 |\tr(hgh^{-1} \kappa(hg^{-1}h^{-1} m, m))|\, dm\, d(h\Zg)
\eeq
converges. Then
\beq{eq def Trg 2}
\int_{G/{\Zg}} \int_M \chi(hgh^{-1}m)^2 \tr(hgh^{-1} \kappa(hg^{-1}h^{-1} m, m))\, dm\, d(h\Zg)
\eeq
is the \emph{$g$-trace} of $\kappa$, denoted by $\Tr_g(\kappa)$.

An operator $T \in \cB(L^2(E))^G$ is $g$-trace class if it has a smooth Schwartz kernel $\kappa$ that is $g$-trace class. Then we write $\Tr_g(T):= \Tr_g(\kappa)$.
\end{definition}

\begin{lemma}\label{lem trace class}
If $G$ is compact, and $T$ is a trace-class operator on $L^2(E)$ with a smooth kernel, then $T$ is $e$-trace class and $\Tr_e(T) = \Tr(T)$.
\end{lemma}
\begin{proof}
If $G$ is compact, we normalise $dg$ so that $G$ has unit volume, and take $\chi \equiv 1$. Then the claim follows from the definitions.
\end{proof}
Lemma \ref{lem trace class} applies for example if $T$ is orthogonal projection onto  the finite-dimensional kernel of an elliptic operator.

\begin{remark}
Definition \ref{def g trace class} is a relatively weak condition. For example, if $G$ is trivial, then an $e$-trace class kernel need not be trace class in the usual sense. 
And, more relevantly to us in part II \cite{HWWII}, it is not necessarily true that a product of $g$-trace class operators is again $g$-trace class. This is why in  \cite{HWWII}, it will take some work   to show that the squares of certain $g$-trace class operators  are again $g$-trace class under certain conditions.

The reason why we use the weak definition of $g$-trace class as in Definition \ref{def g trace class}, is that the operators $S_0$ and $S_1$ in \eqref{eq def Sj} are $g$-trace class in this sense, but not obviously in stronger senses one could think of.
\end{remark}

From now on, fix a $G$-invariant $\Z_2$-grading on $E$. Let $E_+$ be the even part of $E$ for this grading, and $E_-$ the odd part. We will apply Definition \ref{def g trace class} to operators on $L^2(E_{\pm})$ as well.
\begin{definition}\label{def g Fredholm}
Let $D$ be an odd-graded, $G$-equivariant, elliptic differential operator on $E$, 
 Let $D_+$ be its restriction to  sections of $E_+$. Then $D$ is \emph{$g$-Fredholm} if $D_+$ has a parametrix $R$ such that the operators $S_0 :=1_{E_+}-RD_+$ and $S_1 :=1_{E_-}-D_+R$
are $g$-trace class.
\end{definition}

\begin{lemma}\label{lem indep Sj}
Suppose that $D$ is a $g$-Fredholm operator, and let $S_0$ and $S_1$ be as in Definition \ref{def g Fredholm}. Then the number
\[
\Tr_g(S_0) - \Tr_g(S_1)
\]
is independent of the choice of the parametrix $R$.
\end{lemma}

\begin{definition}\label{def g index}
The \emph{$g$-index} of a $g$-Fredholm operator $D$ is the number
\beq{eq def g index}
\ind_g(D) := \Tr_g(S_0) - \Tr_g(S_1),
\eeq
with $S_0$ and $S_1$ as in Definition \ref{def g Fredholm}.
\end{definition}

In the rest of this subsection, we mention some special cases of the $g$-index. Some of these will be used later in this paper, but most are included to show that this index is a natural number to study.

\begin{lemma}\label{lem P}
Suppose that $D$ is a $g$-Fredholm operator. Let $P_{\pm}$ be the orthogonal projections onto the kernels of $D_+$ and $D_+^*$. If $P_+$  and $P_-$ are $g$-trace class, then 
\beq{eq P}
\ind_g(D) = \Tr_g(P_+) - \Tr_g(P_-).
\eeq
\end{lemma}
Lemmas \ref{lem indep Sj} and \ref{lem P} are proved in Subsection \ref{sec prop g index}.

Lemma \ref{lem P} shows that $\ind_e$ generalises Atiyah's $L^2$-index \cite{Atiyah76}, also used in \cite{Connes82, Wang14}.
Ramachandran \cite{Ramachandran} obtained an Atiyah--Patodi--Singer type index theorem, using the right hand side of \eqref{eq P} for $g=e$ to define an index. Lemma \ref{lem P} shows that his index is a special case of the $g$-index. This is also noted  in the proof of Theorem 7.11 on page 343 of \cite{Ramachandran}.

Lemmas \ref{lem trace class} and \ref{lem P} have the following immediate consequence. (We will use this in the discussion of the compact case of the index theorem in Appendix \ref{sec cpt}.)
\begin{lemma} \label{lem fin dim}
If $G$ is compact, and $D$ is an $e$-Fredholm operator with finite-dimensional kernel, then 
\[
\ind_e(D) = \dim(\ker(D_+)) - \dim(\ker(D_+^*)).
\]
\end{lemma}

\begin{lemma}
Suppose that  $D$ is a $g$-Fredholm operator with finite-dimensional kernel, and that $G$ is compact. Let $\chi_{\pm}$ be the characters of the finite-dimensional representations $\ker(D_+)$ and $\ker(D_+^*)$ of $G$. Then
\[
\ind_g(D) = \chi_+(g) - \chi_-(g).
\]
\end{lemma}
\begin{proof}
If $G$ is compact, then one may take $\chi = 1$ in \eqref{eq def Trg 2}. It then follows that $\Tr_g(P_{\pm}) = \chi_{\pm}(g)$, so the claim follows from Lemma \ref{lem P}.
\end{proof}

In part II of this series \cite{HWWII}, we relate the $g$-index to a $K$-theoretic index of $D$. This was done before in cases where $M/G$ is compact. (We will need the case where $M/G$ is noncompact as well, since we apply the $g$-index to the manifold $\hat M$ in \eqref{eq def M hat}.) If $M/G$ is compact, then the natural $K$-theoretic index of $D$ to use the the analytic assembly map \cite{Connes94}. This yields
\[
\ind_G(D) \in K_*(C^*_r(G)).
\]
Here $C^*_r(G)$ is the \emph{reduced group $C^*$-algebra} of $G$: the closure in $\cB(L^2(G))$ of the algebra of left convolution operators by functions in $L^1(G)$.
The relation with the $g$-index is based on the \emph{orbital integral trace}
\beq{eq taug}
\tau_g(f) = \int_{G/{\Zg}} f(hgh^{-1})\, d(h{\Zg}),
\eeq
for functions $f$ on $G$ for which the integral converges. If that is the case for $f$ in a dense subalgebra of $C^*_r(G)$, closed under holomorphic functional calculus, then $\tau_g$ defines a map $\tau_g\colon K_0(C^*_r(G)) \to \C$.
\begin{proposition}\label{prop g index cocpt}
Suppose that $M/G$ is compact, and that either
\begin{itemize}
\item $g=e$;
\item $G$ is discrete and $\tau_g$ extends to a dense subalgebra of $C^*_r(G)$, closed under holomorphic functional calculus; or
\item $G$ is a real semisimple Lie group and $g$ is a semisimple element.
\end{itemize}
Then $D$ is $g$-Fredholm, and
\[
\ind_g(D) = \tau_g(\ind_G(D)).
\]
\end{proposition}
\begin{proof}
The first case is Propositions 3.6 and 4.4 in \cite{Wang14}. The second case is Propositions 3.20 and 5.9 in \cite{Wangwang}. The last case is Lemma 3.5 and Proposition 3.6 in \cite{HW2}.
\end{proof}

In cases where the $g$-index can be recovered from the $K$-theoretic index, as in Proposition \ref{prop g index cocpt} and in the main result in  \cite{HWWII}, homotopy invariance of the latter index implies homotopy invariance of the former. See Corollary 2.9 in \cite{HWWII}.

\subsection{Manifolds with boundary} \label{sec bdry}

We now specialise to the case we are interested in in this paper. Slightly changing notation from the previous subsection, we let $M$ be a Riemannian manifold with boundary $N$. We still suppose that $G$ acts properly and isometrically on $M$, preserving $N$, such that $M/G$ is compact. We assume that a $G$-invariant neighbourhood $U$ of $N$ is $G$-equivariantly isometric to a product $N \times (0,\delta]$, for a $\delta > 0$. To simplify notation, we assume that $\delta = 1$; the case for general $\delta$ is entirely analogous.

As before, let  $E = E_+ \oplus E_- \to M$  be a $\Z_2$-graded $G$-equivariant, Hermitian vector bundle. We assume that $E$ is a \emph{Clifford module}, in the sense that there is a $G$-equivariant vector bundle homomorphism $c$, the Clifford action, from the Clifford bundle of $TM$ to the endomorphism bundle of $E$, mapping odd-graded elements of the Clifford bundle to odd-graded endomorphisms. We also assume that there is a $G$-equivariant isomorphism of Clifford modules $E|_U \cong E|_N \times (0,1]$.

Let $D$ be a Dirac-type operator on $E$, of the form 
\beq{eq Dirac op}
D = c \circ \nabla 
\eeq
for a Hermitian Clifford connection $\nabla$ on $E$.
%
%
Let $D_+$ be the restriction of $D$ to sections of $E_+$. Suppose that
\[
D_+|_U = \sigma \Bigl(-\frac{\partial}{\partial u} + D_N\Bigr),
\]
where $\sigma\colon E_+|_N \to  E_-|_N$ is a $G$-equivariant vector bundle isomorphism, $u$ is the coordinate in the factor $(0,1]$ in $U = N \times (0,1]$, and $D_N$ is an (ungraded) Dirac-type operator on $E_+|_N$. We initially assume that $D_N$ is \emph{invertible}, and show how to remove this assumption in Section \ref{sec DN not inv}.

Consider the cylinder $C := N \times [0,\infty)$, equipped with the product of the metric on $M$ restricted to $N$, and the Euclidean metric. Because the metric, group action, Clifford module and Dirac operator have a product form on $U$, all these structures extend naturally to $C$. 
Let $D_C$ be the extension of $D|_U$ to $C$, acting on sections of
 the extension of $E|_U$ to $C$.
  We form the complete manifold
\beq{eq def M hat}
\hat M := (M \sqcup C)/\sim,
\eeq
where $m \sim (n,u)$ if $m  = (n,u)\in U = N \times (0,1]$. Let $\hat E \to \hat M$ and $\hat D$ be the extensions of $E$ and $D$ to $\hat M$, respectively, obtained by gluing the relevant objects on $M$ and $C$ together along $U$.

\subsection{Delocalised $\eta$-invariants} \label{sec eta}

If $M$, and hence $N$, is compact, then $D_N$ has discrete spectrum. In \cite{APS1}, Atiyah, Patodi and Singer consider the sum
\[
\sum_{\lambda  \in \spec(D_N) \setminus \{ 0\}} \sgn(\lambda) |\lambda|^{-s},
\]
where the sum is over the nonzero eigenvalues of $D_N$ taken with multiplicities, and the real part of $s \in \C$ is large enough so that the sum converges. This sum has a meromorphic continuation to $\C$ as a function of $s$, which is regular at $s=0$. The $\eta$-invariant of $D_N$, denoted by $\eta(D_N)$, is the value of this continuation at $s=0$.

We will denote the smooth kernel of $D_N e^{-tD_N^2}$ by $\tilde \kappa_t$.
Bismut and Freed proved that for twisted $\Spin$-Dirac operators,
there is a function $b \in C^{\infty}(N)$ such that
\beq{eq ass BF}
\tr(\tilde \kappa_t(n,n)) = b(n)t^{1/2}+ O(t^{3/2})
\eeq
as $t \downarrow 0$, uniformly for  $n$ in compact subsets of $N$. See Theorem 2.4 in  \cite{Bismut86}. If $D_N$ has the property \eqref{eq ass BF}, then we have the equivalent definition
\beq{eq eta int}
\eta(D_N) = \frac{2}{\sqrt{\pi}}\int_0^{\infty} \Tr(D_N e^{-t^2D_N})\, dt,
\eeq
where $\Tr$ is the operator trace. See Section II(d) in \cite{Bismut86}.

For general Dirac-type operators $D_N$, the equality \eqref{eq eta int}
 still holds if we replace the lower integration limit by $\varepsilon >0$, asymptotically expand the integral as $\varepsilon \to 0$, and take the right hand side of \eqref{eq eta int} to be the constant term in this expansion. In the compact case, we can take $\chi \equiv 1$,
 so that
 $\Tr_e = \Tr$. (See also Lemma \ref{lem trace class}.)
Hence the following definition is a generalisation of the definition of the $\eta$-invariant in the compact case.
\begin{definition}
The operator $D_N$ \emph{has a $g$-delocalised $\eta$-invariant} if the operator $D_N e^{-tD_N^2}$ is $g$-trace class for all $t>0$, and the integral
\beq{eq ind deloc eta}
 \frac{2}{\sqrt{\pi}} \int_0^{\infty} \Tr_g(D_N e^{-t^2D_N^2})\, dt
\eeq
converges. The value of that integral is then the \emph{$g$-delocalised $\eta$-invariant} of $D_N$, and denoted by $\eta_g(D_N)$.
\end{definition}
For the case of the fundamental group of a compact manifold acting on the manifold's universal cover, delocalised $\eta$-invariants were defined by Lott \cite{Lott92, Lott99}. In that setting, the case for $g=e$ was introduced by Cheeger and Gromov \cite{Cheeger85}, and applied to index theory by Ramachandran \cite{Ramachandran}.

The integral \eqref{eq ind deloc eta} does not always converge, see for example Section 3 of \cite{Piazza07}. However, in the case of fundamental groups of compact manifolds acting on their universal covers, and for $g$ not a central element, a sufficient condition for the convergence of \eqref{eq ind deloc eta} is that $D_N$ has a sufficiently large spectral gap at zero, see Theorem 1.1 in \cite{CWXY19}.
 We discuss this convergence in Section \ref{sec conv eta}, and find concrete classes of situations for discrete and semisimple groups when \eqref{eq ind deloc eta} converges.

Convergence of the integral \eqref{eq ind deloc eta} means convergence of the two integrals
\beq{eq conv eta small large t}
\int_0^1 \Tr_g(D_Ne^{-t^2 D_N^2})\, dt \qquad \text{and} \qquad \int_1^{\infty} \Tr_g(D_Ne^{-t^2 D_N^2})\, dt. 
\eeq
\begin{definition}\label{def eta small large t}
If the first integral in \eqref{eq conv eta small large t} converges, then we say that \emph{$\eta_g(D_N)$ converges for small $t$}. If the second integral in \eqref{eq conv eta small large t} converges, then we say that \emph{$\eta_g(D_N)$ converges for large $t$}. 
\end{definition}

\subsection{The index theorem}

We first state the general form of our index theorem, and then make this more concrete for discrete groups and semisimple Lie groups in Subsection \ref{sec special cases}. In the general form of the index theorem,  Theorem \ref{thm APS}, we list technical conditions that imply a general  index formula. The content of the specific forms for discrete and semisimple groups, Corollaries \ref{cor APS discr} and \ref{cor APS ss}, is that those technical conditions hold in these settings, and that the general index formula then has a concrete topological expression.

Let $X$ be a proper, isometric $G$-manifold, and, for $t>0$, let $\kappa_t$ be the smooth kernel of a $G$-equivariant, $g$-trace class operator on a space of smooth sections of a Hermitian $G$-vector bundle over $X$, such that $\kappa_t$ depends smoothly on $t$. Let $\chi_g \in C^{\infty}(X)$ be nonnegative, and such that for all $m \in X$,
\beq{eq chig cutoff}
\int_{{\Zg}} \chi_g(zm)^2 \, dz = 1.
\eeq
We say that \emph{$\kappa_t$ localises at $X^g$ as $t \downarrow 0$}, if
for all ${\Zg}$-invariant neighbourhoods $V$ of $X^g$,
\beq{eq loc Mg}
\lim_{t\downarrow 0} \int_{X \setminus V} \chi_g(x)^2 |\tr(g\kappa_t(g^{-1}x, x))|\, dx = 0.
\eeq

Consider the setting of Subsection \ref{sec bdry}. Let $\tilde M$ be the double of $M$, and let $\tilde E = \tilde E_+ \oplus \tilde E_-$ and $\tilde D$ be the extensions of $E$ and $D$ to $\tilde M$, respectively.
More explicitly, as on page 55 of \cite{APS1}, $\tilde M$ is obtained from $M$ by gluing together a copy of $M$ and a copy of $M$ with reversed orientation, while $\tilde E$ is obtained by  gluing together a copy of $E$ and a copy of $E$ with reversed grading. To glue these copies of $E$ together along $N$, we use the isomorphism $\sigma$.
Let $\tilde D_{\pm}$ be the restrictions of $\tilde D$ to the sections of $\tilde E_{\pm}$.
\begin{theorem} \label{thm APS}
Suppose that $D_N$ is invertible, and furthermore that 
\begin{itemize}
\item the operators $e^{-t\tilde D^2}$ and $e^{-t\tilde D^2} \tilde D$ are $g$-trace class for all $t>0$;
\item the Schwartz kernels of $e^{-t\tilde D^2}\tilde D$ and $\varphi e^{-tD_C^2}D_C \psi$ localise at $M^g$ as $t \downarrow 0$, for all $G$-invariant $\varphi, \psi \in C^{\infty}(C)$, supported in the relatively cocompact neighbourhood $U \cong N \times (0,1]$ of $N$; and
\item $\eta_g(D_N)$ converges for large $t$.
\end{itemize}
Then operators $\hat D$ and $\tilde D$ are $g$-Fredholm, $D_N$ has a $g$-delocalised $\eta$-invariant, and
\beq{eq APS}
\ind_g(\hat D) = \frac{1}{2}
\ind_g(\tilde D)
- \frac{1}{2}\eta_g(D_N).
\eeq
\end{theorem}
We generalise this result to the case where $D_N$ may not be invertible, but zero is isolated in its spectrum, in Section \ref{sec DN not inv}; see Theorem \ref{thm DN not inv}.

\begin{remark}
The condition in Theorem \ref{thm APS} that either $g$ has no fixed points or $D$ is a twisted $\Spinc$-Dirac operator is used to prove convergence of the integral in \eqref{eq ind deloc eta} for small $t$; see the proof of Proposition \ref{prop proj lim} in Subsection \ref{sec proj lim}.
\end{remark}

\begin{remark}
Reversing both the orientation on $M$ and the grading on $E$ on the second copy of $M$ inside $\tilde M$ will ensure that the topological expression for the $g$-trace of the index of $\tilde D$ is twice the contribution of the interior of $M$ to $\ind_g(\hat D)$. In particular, this choice of orientation and grading implies that the first term on the right hand side of \eqref{eq APS} is not (always) zero.

In fact, the use of the double of $M$ is not essential to our arguments, and we may replace $\tilde M$ by \emph{any} manifold without boundary to which $E$, $D$ and the action by $G$ extend in suitable ways, such that $\tilde M/G$ is compact. Then the first term on the right hand side of \eqref{eq APS} should be replaced by a difference of $g$-traces of heat operators on $\tilde M$ composed with functions supported in $M$. See Lemma \ref{lem S0 prime S1} and its proof.
\end{remark}

For more concrete versions of Theorem \ref{thm APS}, we will show that its conditions hold in relevant cases, and give a topological expression for the difference of $g$-traces on the right hand side of \eqref{eq APS}.
Such a topological expression takes the following form. We specialise to the case of twisted $\Spinc$-Dirac operators. The case of general Dirac-type operators is similar; see for example Section 6.4 in \cite{BGV}. Suppose that $E = S \otimes V$, for the spinor bundle $S \to M$ of a $G$-equivariant $\Spinc$-structure, and a $G$-equivariant, Hermitian vector bundle $V \to M$. Let $R_V$ be the curvature of a $G$-invariant, Hermitian connection on $V$. Let $L \to M$ be the determinant line bundle of $S$.

If $g=e$, then we obtain a result for general unimodular locally compact groups.
\begin{corollary} \label{cor e}
If $D_N$ is invertible, then
\[
\ind_e(\hat D) = \int_{M}  \chi^2 {\hat A(M) e^{c_1(L)/2} \tr(e^{-R_{V}/2\pi i}) } - \frac{1}{2}\eta_e(D_N).
\]
\end{corollary}
This is a generalisation of the main result in \cite{Ramachandran} to arbitrary proper actions by locally compact groups, and a generalisation of the main result in \cite{Wang14} to manifolds with boundary.

For more general elements $g \in G$,
let $M^g$ be the fixed point set of $g$, and $\cN \to M^g$ its normal bundle in $M$. Let $R^{\cN}$ be the curvature of the restriction of the Levi-Civita connection of $M$ to $\cN$. Let $\chi_g \in C_c(M^g)$ be a function with the property \eqref{eq chig cutoff}.
%
As before, we denote all extensions to $\tilde M$ by tildes.
\begin{question} \label{question index}
Under what conditions does the formula
\beq{eq index formula}
\ind_g(\tilde D)
= \int_{\tilde M^g} \tilde \chi_g^2 \frac{\hat A(\tilde M^g) e^{c_1(L|_{\tilde M^g})/2} \tr(ge^{-R_{\tilde V}|_{\tilde M^g}/2\pi i}) }{\det(1-g e^{-R_{\tilde \cN}/2\pi i})^{1/2}}
\eeq
hold?
\end{question}
This question is formulated here for the operator $\tilde D$ on the double $\tilde M$ of $M$; it can be asked more generally of course, for twisted $\Spinc$-Dirac operators on any Riemannian manifold with a proper, cocompact, isometric action.
We will use affirmative answers to this question (the index theorems proved in \cite{HW2, Wangwang}) to make Theorem \ref{thm APS} more concrete in the next subsection.
The fact that the answer is `yes' for the two very different classes of groups we will discuss suggests that \eqref{eq index formula} should hold more generally.


\subsection{Discrete groups and semisimple groups} \label{sec special cases}

The conditions of Theorem \ref{thm APS} are satisfied, and the answer to Question \ref{question index} is `yes', in the two settings we will now describe.

First of all, suppose that $G$ is discrete, and finitely generated. Let $d_G$ be a word length metric on $G$. Let $(g)$ be the conjugacy class of $g$.
\begin{definition}\label{def poly growth}
The conjugacy class $(g)$ has \emph{polynomial growth} if there exist $C,d>0$ such that for all $k \in \N$,
\beq{eq pol growth}
 \#\{h \in (g); d_G(e,h) = k\} \leq Ck^d.
\eeq
\end{definition}
\begin{corollary} \label{cor APS discr}
If $D_N$ is invertible, $G$ is discrete and finitely generated, and $(g)$ has polynomial growth, then a twisted $\Spinc$-Dirac operator $D$ is $g$-Fredholm, and
\[
\ind_g(\hat D)= \int_{M^g} \chi_g^2 \frac{\hat A(M^g) e^{c_1(L|_{M^g})/2} \tr(ge^{-R_{V}|_{M^g}/2\pi i}) }{\det(1-g e^{-R_{\cN}/2\pi i})^{1/2}} - \frac{1}{2}\eta_g(D_N).
\]
\end{corollary}
\begin{remark}\label{rem non poly growth}
The assumption in Corollary \ref{cor APS discr} that $(g)$ has polynomial growth is used to prove large $t$ convergence of $\eta_g(D_N)$, see Proposition \ref{prop deloc eta discr}.
It is possible that this assumption can be removed if one uses other criteria for the convergence of this delocalised $\eta$-invariant, such as the one in \cite{CWXY19}. See Remark \ref{rem conv eta discr}.
\end{remark}


Next, we consider semisimple Lie  groups.
\begin{corollary} \label{cor APS ss} 
Suppose that $D_N$ is invertible.
Let $G$ be a linear, connected, real semisimple Lie group. Let $K<G$ be a maximal compact subgroup. If $G/K$ is even-dimensional, then for all semisimple $g \in G$, the operator $D$ is $g$-Fredholm, and
\beq{eq APS ss}
\ind_g(\hat D)
 = \int_{M^g} \chi_g^2 \frac{\hat A(M^g) e^{c_1(L|_{M^g})/2} \tr(ge^{-R_{V}|_{M^g}/2\pi i}) }{\det(1-g e^{-R_{\cN}/2\pi i})^{1/2}} - \frac{1}{2}\eta_g(D_N).
\eeq
\end{corollary}
%
%

Theorem \ref{thm APS} and Corollaries \ref{cor e}, \ref{cor APS discr} and \ref{cor APS ss} are proved in Subsection \ref{sec pf APS}.
Just like Theorem \ref{thm APS}, Corollaries \ref{cor APS discr} and \ref{cor APS ss} generalise to the case where $D_N$ is not necessarily invertible. See Corollary \ref{cor DN not inv}.

\section{Properties of the $g$-trace and the $g$-index}\label{sec prop trace ind}

\subsection{The trace property}

%

Let $\chi_G \in C^{\infty}(G)$ be any nonnegative function which has compact support on every orbit of the action by ${\Zg}$ on $G$ by right multiplication, and such that for all $x \in G$,
\[
\int_{{\Zg}} \chi_G(xz^{-1})^2\, dz = 1.
\]
Define the nonnegative  function $\chi_g$ on $M$ by
\beq{eq tilde chi}
\chi_g(m)^2 = \int_{G} \chi_G(x)^2 \chi(xgm)^2 \, dx.
\eeq
It has the property \eqref{eq chig cutoff}.
\begin{lemma}\label{lem Trg Tr}
If $T$ is a $g$-trace class operator  on $L^2(E)$ with smooth kernel $\kappa$,
then
\[
\Tr_g(T) =\int_M \chi_g(m)^2\tr(g\kappa(g^{-1}m, m))\, dm.
%
\]
\end{lemma}
\begin{proof}
%
%
%
 For all $h \in G$ and $m \in M$,
\[
\tr(hgh^{-1}\kappa(hg^{-1}h^{-1}m, m)) = \tr(g\kappa(g^{-1}h^{-1}m, h^{-1}m)),
\]
by $G$-invariance of $\kappa$ and the trace property of $\tr$. So, substituting $m'= h^{-1}m$, we find that
\eqref{eq def Trg 2} equals
\begin{multline*}
\int_{G/Z} \int_Z \chi_G(hz^{-1})^2 \int_M \chi(hgh^{-1}m)^2 \tr(hgh^{-1}\kappa(hg^{-1}h^{-1}m, m))\, dm\,\, dz\,  d(h\Zg) \\
= 
  \int_M \chi_g(m')^2\tr(g\kappa(g^{-1}m', m'))\, dm'.
 \end{multline*}
\end{proof}
%
%
%
%

The following trace property is an analogue of Proposition 3.18 in \cite{Wangwang} and Lemma 3.2 in \cite{HW2}.
\begin{lemma}\label{lem Trg trace}
Let $S,T \in \cB(L^2(E))^G$. Suppose that
\begin{itemize}
\item $S$ has a distributional  kernel;
\item $T$ has a smooth kernel;
\item $ST$ and $TS$ are $g$-trace-class.
\end{itemize}
Then $\Tr_g(ST) = \Tr_g(TS)$.
\end{lemma}
\begin{proof}
For notational simplicity, we consider the case where $E$ is the trivial line bundle, so that
$\kappa_T\in C^{\infty}(M\times M)$ and  $\kappa_S\in\mathcal{D}'(M\times M)$.

For $m \in M$, the restriction of $\kappa_S$ to $\{m\} \times M$ is a well-defined distribution in $\mathcal{D}'(M)$, because the wave front set of $\kappa_S$ is normal to the diagonal. We denote this restriction by  $\kappa_{S, m}^{(2)}$. For $m' \in M$, we also write
 $\kappa_{T, m'}^{(1)}
 \in C^{\infty}(M)$ for the restriction of $\kappa_T$ to $M \times \{m'\}$. The pairing between
 $\kappa_{S, m}^{(2)}$ and  $\kappa_{T, m'}^{(1)}$ is
well-defined for all $m,m' \in M$, because the compositions $ST$ and $TS$ are.

By  Lemma \ref{lem Trg Tr},
\[
\Tr_g(ST)
=\int_M \chi_g(m)^2\kappa^{(2)}_{S, g^{-1}m}(\kappa^{(1)}_{T,m})\, dm,
\]
and
\[
\Tr_g(TS)
=\int_M \kappa^{(2)}_{S, g^{-1}m}(\chi_g^2\kappa_{T,m}^{(1)})\, dm.
\]
Using \eqref{eq chig cutoff}, unimodularity of ${\Zg}$, and ${\Zg}$-equivariance of $S$ and $T$, one computes
\[
\begin{split}
\int_M \chi_g(m)^2\kappa^{(2)}_{S, g^{-1}m}(\kappa^{(1)}_{T,m})\, dm&=\int_M \chi_g(m)^2\kappa^{(2)}_{S, g^{-1}m}\Bigl(\int_{{\Zg}}z^{-1}\cdot \chi_g^2\, dz\, \kappa^{(1)}_{T,m}\Bigr)\, dm\\
&=\int_{{\Zg}}\int_M\chi_g(m)^2\kappa^{(2)}_{S, g^{-1}m}\bigl(z^{-1}\cdot (\chi_g^2\kappa^{(1)}_{T,zm})\bigr)\, dm\, dz\\
&=\int_{{\Zg}}\int_M\chi_g(z^{-1}m)^2\kappa^{(2)}_{S, g^{-1}z^{-1}m}\bigl(z^{-1} (\chi_g^2\kappa^{(1)}_{T,m})\bigr)\, dm\, dz\\
&=\int_{{\Zg}}\int_M\chi_g(z^{-1}m)^2\kappa^{(2)}_{S, g^{-1}m}(\chi_g\kappa^{(1)}_{T,m})\, dm\, dz\\
&=\int_M\kappa^{(2)}_{S, g^{-1}m}(\chi_g^2\kappa^{(1)}_{T,m})\, dm.
\end{split}
\]
\end{proof}

\subsection{Properties of the $g$-index}\label{sec prop g index}
Lemma \ref{lem Trg trace} allows us to prove Lemmas \ref{lem indep Sj} and \ref{lem P}.
\begin{proof}[{Proof of Lemma \ref{lem indep Sj}}] 
Let $R$ and $R'$ be parametrices of $D_+$ such that
such that the operators
\[
\begin{split}
S_0 &:=1-RD_+;\\
S_1 &:=1-D_+R;\\
S_0' &:=1-R'D_+;\\
S_1' &:=1-D_+R'\\
\end{split}
\]
are $g$-trace class. Then
\beq{eq Sj Sj'}
\begin{split}
S_0 - S_0' &= (R'-R)D_+;\\ 
S_1 - S_1' &= D_+(R'-R). 
\end{split}
\eeq
The operator $D_+(R'-R)$ is smoothing, because $S_1$ and $S_1'$ are assumed to be. So  elliptic regularity implies that  $R'-R$ is smoothing as well. Hence $R'-R$ has a smooth kernel.
The operators $S_0 - S_0'$ and $S_1 - S_1'$ are $g$-trace class by assumption. The operator $D_+$ has a distributional kernel. 
So Lemma \ref{lem Trg trace} and \eqref{eq Sj Sj'} imply that
\[
(\Tr_g(S_0) - \Tr_g(S_1)) - (\Tr_g(S_0') - \Tr_g(S_1')) = \Tr_g((R'-R)D_+) - \Tr_g(D_+(R'-R)) = 0.
\]
\end{proof}

\begin{proof}[Proof of Lemma \ref{lem P}]
This can be proved as on pages 46--47 of \cite{Atiyah76}, where the trace property of the operator trace is replaced by Lemma \ref{lem Trg trace}.
\end{proof}

\subsection{A trace-like map on Banach algebras}\label{sec TR}

The results and constructions in this subsection are analogues and generalisations of pages 18--19 of \cite{Lott99}. 
The discussion that follows can be simplified if we first consider the case of the trivial line bundle, and then note that one can reduce the general case to this one via local frames. We have chosen to be explicit here, at the cost of some notational inconvenience.

Let $M$ be a \emph{cocompact} $G$-manifold, and $E\to M$ a Hermitian $G$-vector bundle.
Fix an open cover $\{U_j\}$ of open sets $U_j \subset M$ on which $E$ admits local orthonormal frames$\{e_{j,k}\}_{k=1}^r$. Let $\{\zeta_j\}$ be a partition of unity subordinate to $\{U_j\}$. For $\kappa \in \Gamma(\End(E))^G$ and $m_1,m_2\in M$, for  $j_1, j_2$ such that $m_1 \in U_{j_1}$ and $m_2 \in U_{j_2}$, and for
 $k_1, k_2 = 1,\ldots r$,
consider the function $F_{j_1, j_2, k_1, k_2}(\kappa; m_1, m_2)$ on $G$, mapping $x \in G$ to 
\beq{eq def Fjk}
\bigl(e_{j_1, k_1}(m_1), x\kappa(x^{-1}m_1, m_2) e_{j_2, k_2}(m_2)\bigr)_{E_{m_1}}.
\eeq

Let $A$ be a Banach algebra (with respect to convolution)
of functions on $G$. Let $\End_A^0(\Gamma^{\infty}(E))^G$ be the space of 
bounded operators on $L^2(E)$ with 
continuous kernels $\kappa \in \Gamma(\End(E))^G$ such that all of the functions $F_{j_1, j_2, k_1, k_2}(\kappa; m_1, m_2)$ lie in $A$, and such that the maps 
\beq{eq End A cts}
(m_1, m_2) \mapsto F_{j_1, j_2, k_1, k_2}(\kappa; m_1, m_2)
\eeq
 from $M \times M$ to $A$ are continuous.

Let $\chi$ be a cutoff function on $M$, then $\supp(\chi)$ is compact. For $\kappa \in \End_A(\Gamma^{\infty}(E))^G$, set
\beq{eq def norm A}
\|\kappa \|_{A} := \vol(\supp(\chi)) \|\chi\|_{\infty}^2 \max_{m_1,m_2 \in \supp(\chi)} 
\sum_{j_1, j_2} \sum_{k_1, k_2 = 1}^r \zeta_{j_1}(m_1) \zeta_{j_2}(m_2)
\|F_{j_1, j_2, k_1, k_2}(\kappa; m_1, m_2)\|_A.
\eeq
Continuity of the maps \eqref{eq End A cts} implies that the norms of these functions
take maximum  values on the  compact set $\supp(\chi)$. Furthermore,  the sums are locally finite, so $\|\kappa \|_{A} $ is well-defined. Different choices of a partition of unity and local orthonormal frames result in equivalent norms, but we will not need this.

\begin{lemma}\label{lem submult}
The norm $\|\cdot  \|_{A}$ is submultiplicative.
\end{lemma}
\begin{proof}
Let  $\kappa,\lambda \in \End_A^0(\Gamma^{\infty}(E))^G$. Then 
for all $m_1,m_2 \in M$ and $x \in G$,
\[
\begin{split}
x(\kappa \lambda)(x^{-1}m_1,m_2) &=
\int_M \int_G \chi(ym_3)^2 
x\kappa(x^{-1}m_1, m_3) \lambda(m_3, m_2)\, dy\, dm_3
 \\
&=
\int_M \chi(m_3)^2 \int_G
xy^{-1}\kappa(yx^{-1}m_1, m_3) y\lambda(y^{-1}m_3, m_2)\,dy\,  dm_3.
%
 \end{split}
 \]
For the first equality, we used the definition of cutoff functions. For the second, we substituted $ym_3$ for $m_3$ and used $G$-invariance of $\kappa$. 

So for all relevant $j_1, j_2, k_1$ and $k_2$, and $x \in G$,
\begin{multline*}
F_{j_1, j_2, k_1, k_2}(\kappa \lambda; m_1, m_2)(x) = \\
\int_M \chi(m_3)^2 \int_G
\left(e_{j_1, k_1}(m_1), 
xy^{-1}\kappa(yx^{-1}m_1, m_3) y\lambda(y^{-1}m_3, m_2)
 e_{j_2, k_2}(m_2)\right)_{E_{m_1}}
 \,dy\,  dm_3 = \\
 \sum_{j_3}\sum_{k_3 = 1}^r
 \int_M \chi(m_3)^2  \zeta_{j_3}(m_3) \int_G
\left(e_{j_1, k_1}(m_1), 
xy^{-1}\kappa(yx^{-1}m_1, m_3) e_{j_3, k_3}(m_3) \right)_{E_{m_1}}\\
\left(
e_{j_3, k_3}(m_3)
y\lambda(y^{-1}m_3, m_2)
 e_{j_2, k_2}(m_2)\right)_{E_{m_3}}
 \,dy\,  dm_3 = \\
  \sum_{j_3}\sum_{k_3 = 1}^r
 \int_M \chi(m_3)^2 \zeta_{j_3}(m_3)
 \bigl(F_{j_1, j_3, k_1, k_3}(\kappa; m_1, m_3) *  F_{j_3, j_2, k_3, k_2}(\lambda; m_3, m_2) \bigr)(x) 
 \,  dm_3. 
\end{multline*}

By submultiplicativity of the norm $\|\cdot \|_A$ on $A$, we find that 
\begin{multline*}
\| F_{j_1, j_2, k_1, k_2}(\kappa \lambda; m_1, m_2)\|_A \leq\\
  \sum_{j_3}\sum_{k_3 = 1}^r
 \int_M \chi(m_3)^2 \zeta_{j_3}(m_3)
 \| F_{j_1, j_3, k_1, k_3}(\kappa; m_1, m_3) \|_A \|  F_{j_3, j_2, k_3, k_2}(\lambda; m_3, m_2) \|_A 
 \,  dm_3 \leq\\
 \vol(\supp(\chi)) \|\chi\|_{\infty}^2 \max_{m_3, m_4 \in \supp(\chi)}
   \sum_{j_3, j_4}\sum_{k_3, k_4 = 1}^r \zeta_{j_3}(m_3) \zeta_{j_4}(m_4)\\
 \| F_{j_1, j_3, k_1, k_3}(\kappa; m_1, m_3) \|_A \|  F_{j_4, j_2, k_4, k_2}(\lambda; m_4, m_2) \|_A. 
\end{multline*}
This implies the claim by the definition \eqref{eq def norm A} of the norm $\|\cdot \|_A$.
%
%
\end{proof}
Let $\End_A(\Gamma^{\infty}(E))^G$ be the completion of $\End_A^0(\Gamma^{\infty}(E))^G$ (that space may already be complete in some cases). By Lemma \ref{lem submult},  $\End_A(\Gamma^{\infty}(E))^G$  is a Banach algebra.

\begin{lemma}\label{lem TR}
Setting
\beq{eq def TR}
\TR(\kappa)(x) = \int_M \chi(xm)^2\tr(x\kappa(x^{-1}m, m))\, dm,
\eeq
for $\kappa \in \End_A^0(\Gamma^{\infty}(E))^G$ and $x \in G$, defines a continuous map  $\TR\colon \End_A(\Gamma^{\infty}(E))^G \to A$ with respect to the norm $\|\cdot \|_A$.
\end{lemma}
\begin{proof}
Let $\kappa \in \End_A^0(\Gamma^{\infty}(E))^G$. Then by a substitution $xm \mapsto m$, equivariance of $\kappa$ and the trace property, we have for all $x \in G$,
\[
\begin{split}
\TR(\kappa)(x) &= \int_{M} \chi(m)^2 \tr(x\kappa(x^{-2}m, x^{-1}m)) \, dm\\
&=  \int_{M} \chi(m)^2 \tr(x\kappa(x^{-1}m, m)) \, dm \\
&= \sum_{j}\sum_{k=1}^r \int_M
 \chi(m)^2 \zeta_j(m) 
 F_{j,j,k,k}(\kappa; m,m)(x)
 \, dm.
\end{split}
\]
So
\[
\begin{split}
\|\TR(\kappa)\|_A&\leq
\sum_{j}\sum_{k=1}^r \int_M
 \chi(m)^2 \zeta_j(m) 
 \| F_{j,j,k,k}(\kappa; m,m)\|_A
 \, dm \\
 &\leq \vol(\supp(\chi)) \|\chi\|_{\infty}^2 \max_{m \in \supp(\chi)}\sum_{j}\sum_{k=1}^r
\zeta_j(m) 
 \| F_{j,j,k,k}(\kappa; m,m)\|_A \\
 &\leq \|\kappa\|_A.
\end{split}
\]
\end{proof}

The motivation for Lemmas \ref{lem submult} and \ref{lem TR} is the following. 
Suppose that $\tau_g$, defined by \eqref{eq taug}, defines a continuous functional on $A$.
It is immediate from the definitions that
\beq{eq Trg taug TR}
\Tr_g = \tau_g \circ \TR.
\eeq
So by Lemma \ref{lem TR}, $\Tr_g$
is a continuous functional on $\End_A(\Gamma^{\infty}(E))^G$.

\begin{remark}
The map $\TR$ generalises the map $\Tr_N$ in (3.11) in \cite{HW2} and the map $\TR$ on page 429 of \cite{Lott92b}.
\end{remark}

\section{Convergence of delocalised $\eta$-invariants} \label{sec conv eta}

%
 
 In the proof of Theorem \ref{thm APS}, we assume large $t$ convergence of $\eta_g(D_N)$, and deduce small $t$ convergence from the index formula. In this section, we study large $t$ convergence, which will be used to verify the condition in Theorem \ref{thm APS} in the special cases in Corollaries \ref{cor e}, \ref{cor APS discr} and \ref{cor APS ss}. We also comment on small $t$ convergence of $\eta_g(D_N)$ in Subsection \ref{sec eta small t}, but these results are only used in the case where $D_N$ is not invertible in Section \ref{sec DN not inv}.
 
 Building on arguments by Lott \cite{Lott99}, we show that the integrand in \eqref{eq ind deloc eta}  is rapidly decreasing as $t$ goes to infinity, if an algebra
 $\cA \subset C^*_r(G)$ with certain properties exists. We then show that such algebras exist for discrete groups and semisimple Lie groups.

\subsection{Projective limits of Banach algebras and large $t$ convergence}\label{sec proj lim}

%

 Suppose that $\cA$ is the projective limit of a sequence of Banach algebras $A_j \subset C^*_r(G)$ with norms $\|\cdot \|_j$, on which the orbital integral trace \eqref{eq taug} defines a continuous functional.
Let $ \End_{A_j}(\Gamma^{\infty}(E|_N))^G$, $\tau_g$ and $\TR$ be as in Subsection \ref{sec TR}. Let
$\End_{\cA}(\Gamma^{\infty}(E|_N))^G$ 
  be the projective limit of the algebras $\End_{A_j}(\Gamma^{\infty}(E|_N))^G$.
In this subsection, we prove the following extension of Proposition 8 in \cite{Lott99}.
\begin{proposition} \label{prop proj lim}
Suppose that $\cA$ is closed under holomorphic functional calculus. 
If $D_N$ is invertible and 
$D_Ne^{-tD_N^2} \in \End_{\cA}(\Gamma^{\infty}(E|_N))^G$ for all $t>0$,
then $\eta_g(D_N)$ converges for large $t$.
\end{proposition}
In Subsections \ref{sec eta discr} and \ref{sec eta ss}, we give examples of cases where the conditions of this proposition hold. In fact, in those examples, the algebra $\cA$ is a Banach algebra itself, rather than just a projective limit of Banach algebras. We have included the possibility that $\cA$ is only a projective limit of Banach algebras for the sake of generality, for example to include the algebras used in \cite{Lott99}.

We prove Proposition \ref{prop proj lim} by extending the proof of Proposition 8 in \cite{Lott99} to our setting.

%


\begin{lemma} \label{lem rapid decr}
Suppose that $\cA$ is the projective limit of a sequence of Banach algebras, and that $D_Ne^{-t^2 D_N^2} \in \End_{\cA}(\Gamma^{\infty}(E|_N))^G$.
Then $\Tr_g(D_Ne^{-t^2 D_N^2}))$ decreases at least as fast as a Gaussian function in $t$ as $t \to \infty$.
\end{lemma}
\begin{proof}
Continuity of $\tau_g$ on $\cA$ means that it is continuous on $A_j$ for all $j$. So, by Lemma \ref{lem TR}, the composition $\Tr_g = \tau_g \circ \TR$ is a continuous functional on $\End_{\cA}(\Gamma^{\infty}(E))^G$.
 By construction of  projective limits, this continuity implies that there exist $j$ and $C>0$ such that for all $t$,
\beq{eq taug TR cts}
\tau_g(\TR(D_Ne^{-t^2 D_N^2})) \leq C \|D_Ne^{-t^2 D_N^2} \|_j
\eeq
Now Lemma \ref{lem submult} implies that
\beq{eq submult}
\|D_Ne^{-t^2D_N^2}\|_j \leq \|D_Ne^{-D_N^2}\|_j \|e^{(t^2-1)D_N^2}\|_j.
\eeq

Since $D_N$ is invertible, the spectrum of $D_N^2$ has a positive smallest element $\lambda_0$.
 By the spectral mapping theorem,  the spectral radius of $e^{-tD_N^2}$ as an operator on $L^2(E)$ is $e^{-t\lambda_0}$. This is in fact also the spectral radius $\rho_{A_j}(e^{-tD_N^2})$ of $e^{-tD_N^2}$ as an element of the Banach algebra $\End_{A_j}(\Gamma^{\infty}(E))^G$. This follows from the fact that $\cA$ is closed under holomorphic functional calculus; the details of this argument are as in  Proposition 19 on  page 38 of \cite{Lott99} and Lemmas 2 and 3 on pages 40--41 of \cite{Lott99}.


By Theorem 1.22 in \cite{Davies} (a general fact about one-parameter semigroups in Banach algebras), the limit
\[
a:= \lim_{t\to \infty} t^{-1} \ln \|e^{-tD_N^2}\|_j
\]
exists, and $\rho_{A_j}(e^{-tD_N^2}) = e^{at}$ for all $t$. So $a = -\lambda_0$. Hence there is a $t_0>0$ such that for all $t>t_0$,
\[
t^{-1} \ln \|e^{-tD_N^2}\|_j < -\lambda_0/2.
\]
Hence for all such $t$,
\[
\|e^{-tD_N^2}\|_j < e^{-t\lambda_0/2}.
\]
(Compare also Corollary 3 on page 41 of \cite{Lott99}.)
We conclude that if $t^2-1 > t_0$,
\[
\|D_Ne^{-t^2D_N^2}\|_j < \|D_Ne^{-D_N^2}\|_je^{-(t^2-1)\lambda_0/2},
\]
 a Gaussian function
 in $t$.
\end{proof}

Proposition \ref{prop proj lim} follows immediately from Lemma \ref{lem rapid decr}.

\begin{remark} \label{rem conv eta non inv}
Lemma \ref{lem rapid decr}, and hence Proposition \ref{prop proj lim}, generalise to the case where $D_N$ is not invertible, but zero is isolated in its spectrum. This can be proved by considering the compression of $D_N$ to the orthogonal complement to its kernel, as on page 20 of \cite{Lott99}. 
%
\end{remark}

\subsection{Small $t$ convergence} \label{sec eta small t}

Small $t$ convergence of $\eta_g(D_N)$ will be deduced from the index formula in the proof of Theorem \ref{thm APS}. But in this subsection, we give independent proofs of small $t$ convergence in some cases. The material in this subsection is only used in Section \ref{sec DN not inv} to replace a regularised delocalised $\eta$-invariant by a usual  delocalised $\eta$-invariant, and even there these direct proofs can probably be avoided (see Remark \ref{rem eta reg direct}).

\begin{proposition} \label{prop proj lim small t}
\begin{enumerate}
\item[(a)]
If $g$ has no fixed points in $N$, then  $\eta_g(D_N)$ converges for small $t$.
\item[(b)]
Without the assumption that $g$ has no fixed points in $N$, 
suppose that either
\begin{enumerate}
\item[(b1)] $G$ is any unimodular locally compact group, and $G/Z$ is compact;
\item[(b2)] $G$ is discrete and finitely generated; or 
\item[(b3)] $G$ is a connected, real semisimple Lie group, and $g$ is a semisimple element.
\end{enumerate}
Suppose furthermore that 
$D_N$ is invertible and a twisted $\Spinc$-Dirac operator, and  that
$D_Ne^{-tD_N^2} \in \End_{\cA}(\Gamma^{\infty}(E|_N))^G$ for all $t>0$.
Then $\eta_g(D_N)$ converges for small $t$, and hence $D_N$ has a $g$-delocalised $\eta$-invariant.
\end{enumerate}
\end{proposition}

From now on, we denote the Schwartz kernel of $D_N e^{-tD_N^2}$ by $\tilde \kappa_t$.
The Bismut--Freed estimate \eqref{eq ass BF} has an equivariant generalisation due to Zhang.
\begin{proposition} \label{prop Zhang}
Suppose that $D_N$ is a twisted $\Spinc$-Dirac operator. For all compact subsets $K_1 \subset G$ and $K_2 \subset N$, there is a $C>0$ such that for all $x \in K_1$ and $t \in (0,1]$,
\beq{eq Zhang}
\left|\int_{K_2} \tr(x\tilde \kappa_t(x^{-1}n, n))\, dn\right|  \leq C\sqrt{t}.
\eeq
\end{proposition}
\begin{proof}
By (2.2) in  \cite{Zhang90}, we can select a constant $C_x>0$ for every $x \in G$ such that \eqref{eq Zhang} holds with $C$ replaced by $C_x$.
In  \cite{Zhang90},  compact manifolds are considered, but it is shown that the computation localises near the fixed point set of $g$.  The same arguments apply with the fixed point set in a compact manifold replaced by the intersection of the fixed point set of $g$ with the compact set $K_2$.

By going through the proof of (2.2) in  \cite{Zhang90}, one sees that the constants in the estimates that are used (in particular, the constant $C_1$ in Lemma 2.17 in \cite{Zhang90}) can be chosen continuously in $x$, so that we can choose a single $C$ such that \eqref{eq Zhang} holds for all $x \in K_1$.
\end{proof}

\begin{lemma}\label{lem int small t cpt A}
Suppose that 
 $D_N$ is a twisted $\Spinc$-Dirac operator. Then
for any compact subset $A\subset G/Z$, the integral
\beq{eq int small t cpt A}
\int_0^1 \int_A \int_N \chi(hgh^{-1}n)^2 \tr(hgh^{-1} \tilde \kappa_t(hg^{-1}h^{-1}n,n))\, dn\, d(hZ)\, dt
\eeq
converges absolutely.
\end{lemma}
\begin{proof}
By compactness of $A$ and $\supp(\chi)$, we can apply Proposition \ref{prop Zhang} to find a $C>0$ such that for all $h \in A$,
\[
\left|
 \int_N \chi(hgh^{-1}n)^2 \tr(hgh^{-1} \tilde \kappa_t(hg^{-1}h^{-1}n,n))\, dn\right| \leq C\sqrt{t}.
\] 
So the integrand in the outer integrals over $[0,1]$ and $A$ in \eqref{eq int small t cpt A} is bounded.
%
%
\end{proof}

\begin{lemma}\label{lem int small t discr}
Suppose that $G$ is discrete and finitely generated. Then there is a compact (i.e.\ finite) subset $A \subset G/Z$ such that  the integrals and sum
\beq{eq int small t discr}
\int_0^1 \sum_{h \in (G/Z) \setminus A} \int_N \chi(hgh^{-1}n)^2 \tr(hgh^{-1} \tilde \kappa_t(hg^{-1}h^{-1}n,n))\, dn\, dt
\eeq
converge absolutely.
\end{lemma}
\begin{proof}
Let $l$ be the word length function for a fixed, finite generating set of $G$. 
By the \v{S}varc--Milnor lemma, there are a finite subset $\tilde A \subset G$ and a $c>0$ such that for all $x \in G \setminus \tilde A$ and $n \in \supp(\chi)$, 
\beq{eq SM A}
d(xn, n) \geq c l(x).
\eeq
Furthermore, by standard Gaussian estimates for heat kernels (see e.g.\ Proposition 4.2(i) in \cite{CGRS14} or (31) in \cite{Lott92}), there are $a_1, a_2, a_3 >0$ such that for all $t \in (0,1]$ and $n,n' \in N$,
\beq{eq Gauss hk}
\|\tilde \kappa_t (n,n')\| \leq a_1 t^{-a_2} e^{-a_3 d(n,n')^2/t}.
\eeq
By \eqref{eq SM A} and \eqref{eq Gauss hk}, 
\beq{eq int small t discr 2}
\begin{split}
\int_0^1 \sum_{x \in G \setminus A} \int_N \chi(x n)^2 |\tr(x \tilde \kappa_t(x^{-1}n,n))| \, dn\, dt &\leq
a_1 
\int_0^1  t^{-a_2}  \sum_{k=1}^{\infty} \sum_{\substack{x \in G \setminus A\\ k\leq l(x) < k-1}}
 \int_N \chi(x n)^2 \, dn\,  e^{-c a_3 l(x)^2/t}    \, dt\\
 &= a_1  \int_N \chi(n)^2 \, dn \int_0^1  t^{-a_2}  \sum_{k=1}^{\infty} \bigl(\#\{x \in G; k\leq l(x) < k-1\}\bigr)
 e^{-c a_3 k^2/t}    \, dt.
 \end{split}
\eeq
Now there is a $b>0$ such that for every $k\geq 1$, 
\[
\int_0^1  t^{-a_2} 
 e^{-c a_3 k^2/t}    \, dt \leq be^{-c a_3 k^2}.
\]
(This is a calculus exercise, see  Lemma 4.2 in \cite{HWWII}.) Because the number $\#\{x \in G; k\leq l(x) < k-1\}$ grows at most exponentially in $k$, we find that 
\[
\sum_{k=1}^{\infty} \bigl(\#\{x \in G; k\leq l(x) < k-1\}\bigr)
\int_0^1  t^{-a_2}  
 e^{-c a_3 k^2/t}    \, dt
\]
converges absolutely. This implies that the left hand side of \eqref{eq int small t discr 2} converges absolutely by Tonelli's theorem. By restricting the sum to the conjugacy class of $g$, we find that \eqref{eq int small t discr} converges absolutely as well.
\end{proof}

\begin{proof}[Proof of Proposition \ref{prop proj lim small t}]
In case (a), Proposition 4.25 in \cite{PPST} and Lemma \ref{lem Trg Tr} 
 imply that the integrand in \eqref{eq ind deloc eta} decreases like a function of the form $t^{-a}e^{-b/t}$ as $t \downarrow 0$, for $a,b>0$. Hence we have convergence of  \eqref{eq ind deloc eta} in case (a).

Convergence of \eqref{eq ind deloc eta}  for small $t$ follows directly from Lemma \ref{lem int small t cpt A} in case (b1) of Proposition \ref{prop proj lim}. In case (b2), this follows from Lemmas \ref{lem int small t cpt A} and \ref{lem int small t discr}. In case (b3), this follows from Proposition 4.29 in \cite{PPST}.
%
%
%
\end{proof}

\begin{remark}
In Proposition 4.29 in \cite{PPST}, used in the proof of case (b3) of Proposition \ref{prop proj lim small t}, it is not assumed that $D_N$ is invertible. So this assumption is not necessary in that case.
\end{remark}

\subsection{Ka\v{g}an Samurka\c{s}'  algebra}

In this subsection and the next, we assume that  $G$ is a finitely generated discrete group, and that $( g )$ has polynomial growth (Definition \ref{def poly growth}) from now on. We briefly review Ka\v{g}an Samurka\c{s}' construction  of an algebra on which the trace $\tau_g$ defines a continuous functional.

For $S \subset G$ and $f \in l^2(G)$, let $\| f \|_S$ be the norm of $f|_S$ in $l^2(S)$. For $r>0$, write
\[
B_r(S) := \{h \in G; \forall s \in S, d_G(e,hs^{-1})<r\}.
\]
For $a \in \cB(l^2(G))$ and $r>0$, define
\[
\mu_a(r) := \inf\{C>0; \forall f \in l^2(G), \|af\|_{G \setminus B_r(\supp(f))} \leq C \|f\|_G\}.
\]
Fix $C>0$ and $d \in \N$ such that \eqref{eq pol growth} holds for all $k$. For $a \in \cB(l^2(G))$, define
\beq{eq norm g}
\|a\|_{g} := \inf\{ A>0; \forall r>0, \mu_a(r) \leq Ar^{-(d/2+2)} \}.
\eeq
Let $\|\cdot \|_{\cB(l^2(G))}$ be the operator norm on $\cB(l^2(G))$. For $f \in \C G$, we also denote the convolution operator by $f$ from the left by $f \in \cB(l^2(G))$
\begin{lemma} \label{lem norm pol}
The expression \eqref{eq norm g} defines a seminorm on $\C G$, and for all $a,b \in \C G$,
\[
\|ab\|_{g} \leq 2^{d/2+2} (\|a\|_{\cB(l^2(G))} + \|a\|_g)(\|b\|_{\cB(l^2(G))} + \|b\|_g).
\]
\end{lemma}
\begin{proof}
The first claim is Lemma 3.6 in \cite{KS17}. In the proof of Lemma 3.10 in \cite{KS17}, it is proved that for all $a,b \in \C G$,
\begin{multline*}
\|ab\|_{g} \leq \max\bigl( 2^{d/2 + 3} \|a\|_{\cB(l^2(G))} \|b\|_g +  2^{d/2 + 2} \|a\|_g \|b\|_{\cB(l^2(G))} + 2^{d/2 + 2} \|a\|_g \|b\|_{g}, \\
\|a\|_{\cB(l^2(G))} \|b\|_{\cB(l^2(G))}\bigr).
\end{multline*}
This implies the second claim.
\end{proof}
 Let $C_g^{\pol}(G)$ be the completion of $\C G$ in the norm
 \[
 \|f\| := \|f\|_{\cB(l^2(G))} + \|f\|_g.
 \]
 By Lemma \ref{lem norm pol}, this is a Banach algebra. By Lemma 3.12 in \cite{KS17}, it is closed under holomorphic functional calculus. Theorems 3.18 and 3.19 in \cite{KS17} state that $\tau_g$ defines a continuous trace on $C_g^{\pol}(G)$.

We will use the following equality.
\begin{lemma}\label{lem mu f}
For all $ f\in l^1(G)$,
\[
\mu_f(r) = \|f\|_{G\setminus B_r(e)}.
\]
\end{lemma}
\begin{proof}
Let $f\in l^1(G)$. Since $G$ is discrete, we also have $f \in l^2(G)$. For all $\varphi \in \C G$, we have
\[
\begin{split}
\|f*\varphi \|_{G \setminus B_r(\supp(\varphi))}^2 &=
\sum_{x \in G\setminus B_r(\supp \varphi) } \Bigl| \sum_{y \in \supp \varphi} f(xy^{-1})\varphi(y) \Bigr|^2 \\
&\leq  \sum_{x \in G\setminus B_r(\supp \varphi) }  \sum_{y \in \supp \varphi} |f(xy^{-1})|^{2} |\varphi(y)|^2  \\
&\leq \sum_{z \in G\setminus B_r(e)} |f(z)|^2  \sum_{y \in \supp \varphi}  |\varphi(y)|^2.
\end{split}
\]
Here we used the fact that for all $x \in G\setminus B_r(\supp \varphi)$ and $y \in \supp \varphi$, we have $xy^{-1} \in G\setminus B_r(e)$. So $\mu_f(r) \leq \|f\|_{G\setminus B_r(e)}$.

And if $\varphi = \delta_e$, then
\[
\|f*\varphi \|_{G \setminus B_r(\supp(\varphi))} = \|f\|_{G\setminus B_r(e)}.
\]
So $\mu_f(r) \geq \|f\|_{G\setminus B_r(e)}$.
\end{proof}

\subsection{Discrete groups} \label{sec eta discr}

We continue to assume  that  $G$ is a finitely generated discrete group, and that $( g )$ has polynomial growth.
From now on, we denote the Schwartz kernel of $e^{-tD_N^2}$ by $\kappa_t$, and recall that we denote the Schwartz kernel of $D_N e^{-tD_N^2}$ by $\tilde \kappa_t$.
%
We first show that the integrand of \eqref{eq ind deloc eta} is well-defined. In the proof of the following lemma, and in other places, we denote the operator norm of an endomorphism $A$ of a finite-dimensional normed vector space by $\|A\|$.
\begin{lemma}\label{lem discr g trace cl}
For all $t>0$,
 $D_Ne^{-tD_N^2}$ is $g$-trace class.
\end{lemma}
\begin{proof}
It is noted  on page 11 of \cite{Lott99} that for all $m \in M$,
\[
\sum_{x \in G} \| x\tilde \kappa_t(x^{-1} m, m)\| < \infty.
\]
In \cite{Lott99} the action considered is free; the argument extends to the proper case because stabilisers are finite. See Theorem 3.4 in \cite{Wangwang}, the proof of which also applies to $D_N e^{-tD_N^2}$, via the inequality in the third displayed equation on page 307 of \cite{Wangwang}. (That theorem is an analogue of Theorem 4.3 in \cite{Donnelly} for non-scalar Laplacians, which also applies to derivatives of heat kernels as noted at the bottom of page 491 in \cite{Donnelly}.)
Hence also
\[
\sum_{x \in (g)} \|x \tilde \kappa_t(x^{-1} m, m)\| < \infty.
\]
In this case, the integral over $G/{\Zg}$ is just the sum over $(g)$ and the integral over $M$ with weight function  $\chi^2$ is the integral over a compact fundamental domain. Hence the right hand side of   \eqref{eq def Trg} converges.
\end{proof}

%

\begin{proposition} \label{prop heat pol}
For all $t >0$, the Schwartz kernels $\kappa_t$
and $\tilde \kappa_t$ 
 lie in $\End_{C_g^{\pol}(G)}(\Gamma^{\infty}(E|_N))^G$.
\end{proposition}

\begin{lemma} \label{lem rapid decr pol}
Fix  $t>0$ and $n,n' \in N$, and consider the functions $\psi := F_{j_1, j_2, k_1, k_2}(\kappa_t; n, n')$ and $\tilde \psi := F_{j_1, j_2, k_1, k_2}(\tilde \kappa_t; n, n')$ defined as in \eqref{eq def Fjk}.
The functions  $\mu_{\psi}$ and $\mu_{\tilde \psi}$ decrease faster than any rational function.
\end{lemma}
\begin{proof}
The function $\psi$ lies in $l^1(G)$; see for example Section 3.2, page 308 of \cite{Wangwang}. And so does $\tilde \psi$, as pointed out in the proof of Lemma \ref{lem discr g trace cl}.
%
%
So by Lemma \ref{lem mu f}, the claim is that the numbers $\|\psi \|_{G\setminus B_r(e)}$ and $\|\tilde \psi \|_{G\setminus B_r(e)}$ decrease faster than any rational function of $r$, as $r \to \infty$.

For $k \in \N$, write
\[
G(k) := \{x \in G; k-1 < d(x^{-1}m,m) \leq k\}.
\]
It was shown on page 307 of \cite{Wangwang} that there are $A_1,b_1 > 0$ such that for all $k \in \N$,
\beq{eq nk}
\# G(k) \leq A_1 e^{b_1k}.
\eeq

Furthermore, for fixed $t>0$, standard estimates for heat kernels show that there are $A_2, b_2> 0$
such that for all $m_1, m_2 \in M$,
\[
\| \kappa_t(m_1, m_2)\| \leq A_2 e^{-b_2d(m_1, m_2)^2}.
\]
Similar estimates hold for the derivatives of $ \kappa_t$. This implies that there are $A_3, b_3> 0$
such that for all $m_1, m_2 \in M$,
\beq{eq est kappa}
\| \tilde \kappa_t(m_1, m_2)\| \leq A_3 e^{-b_3d(m_1, m_2)^2}.
\eeq

From now on, we will consider $\tilde \psi$ and $\tilde \kappa_t$. The same arguments apply to $\psi$ and $\kappa_t$.
For fixed $r>0$, we have
\[
\|\tilde \psi \|_{G \setminus B_r(e)}^2 \leq \sum_{k=1}^{\infty} \sum_{x \in G(k); d_G(e,x)>r} \|x\tilde \kappa_t(x^{-1}n,n')\|^2.
\]
By \eqref{eq est kappa} and the triangle inequality, this is at most equal to
\beq{eq psi 1}
 A_3^2 e^{-2b_3d(n, n')^2}  \sum_{k=1}^{\infty} \#\{x \in G(k); d_G(e,x)>r\}
 e^{-2b_3k^2}.
\eeq

By the \v{S}varc--Milnor Lemma (see e.g.\ Lemma 2 in \cite{Milnor}, or \cite{BDM}), there are constants $A_4, b_4 > 0$ such that for all $x \in G$,
\[
d(x^{-1}n,n) \geq A_4 d_G(e,x)- b_4.
\]
So if $x \in G(k)$ and $d_G(e,x) > r$, then $k \geq A_4 r - b_4$. Together with \eqref{eq nk}, this implies that \eqref{eq psi 1} is at most equal to
\[
 A_1 A_3^2 e^{-2b_3d(n, n')^2}  \sum_{k=\lfloor A_4 r- b_4 \rfloor}^{\infty}
 e^{b_1 k-2b_3k^2}.
\]
This decreases exponentially as $r \to \infty$.
\end{proof}

 \begin{proof}[Proof of Proposition \ref{prop heat pol}]
 Under the conditions in the proposition,
the functions $\psi$ and $\tilde \psi$ decrease faster than any rational function by Lemma \ref{lem rapid decr pol}. So in particular, there are constants $A, \tilde A > 0$ such that for all $r>0$,
\[
\begin{split}
\mu_{\psi}(r) &\leq Ar^{-(d/2+2)};\\
\mu_{\tilde \psi}(r) &\leq \tilde Ar^{-(d/2+2)}.
\end{split}
\]
Hence $\|\psi\|_g$ and $\|\tilde \psi\|_g$ are finite. 
%
We saw in the proof of Lemma \ref{lem rapid decr pol} that $\psi, \tilde \psi \in l^1(G) \subset C^*_r(G)$. There we also saw that the constant $A$ in \eqref{eq norm g} can be chosen continously in $n$ and $n'$.
This finishes the proof.
 \end{proof}

Propositions \ref{prop proj lim}, \ref{prop proj lim small t} and \ref{prop heat pol}
 have the following consequence.
\begin{proposition}\label{prop deloc eta discr}
Suppose that $G$ is a discrete, finitely generated group and $(g)$ has polynomial growth.
If $D_N$ is invertible, then $\eta_g(D_N)$ converges for large $t$. If, furthermore, 
 either $g$ has no fixed points or $D$ is a twisted $\Spinc$-Dirac operator, 
then $\eta_g(D_N)$ also converges for small $t$, so $D_N$ has a $g$-delocalised $\eta$-invariant.
\end{proposition}
\begin{remark}\label{rem conv eta discr}
For finitely generated discrete groups, a general criterion for convergence of delocalised $\eta$-invariants was obtained in \cite{CWXY19}. For hyperbolic groups,  Puschnigg \cite{Puschnigg09} shows that an algebra $\cA$ to which $\tau_g$ extends exists. It may be possible to use this to obtain a version of Proposition \ref{prop deloc eta discr} for such groups.
\end{remark}

\subsection{Real semisimple groups} \label{sec eta ss}

Suppose, in this subsection, that $G$ is a real, linear, connected semisimple Lie group, and that $g \in  G$ is semisimple (i.e.\ $\Ad(g)$ is diagonalisable).
We continue to denote write $\kappa_t$ for the Schwartz kernel of $e^{-tD_N^2}$, and $\tilde \kappa_t$ for the Schwartz kernel of $D_N e^{-tD_N^2}$.
\begin{lemma}\label{lem ss g trace cl}
If $G$ is a real, linear, connected, semisimple Lie group and $g$ is a semisimple element, then $D_Ne^{-tD_N^2}$ is $g$-trace class for all $t>0$.
\end{lemma}
\begin{proof}
By Abels' theorem \cite{Abels}, connectedness of $G$ implies that $N$ decomposes $G$-equivariantly as $N = G\times_K Y$, for a maximal compact subgroup $K<G$ and a $K$-invariant submanifold $Y \subset N$. Then $\kappa_t$  lies in
\[
\bigl( C^{\infty}(G\times G) \otimes \Gamma^{\infty}( \End(E|_Y) )\bigr)^{G\times K} \cong
\bigl( C^{\infty}(G) \otimes \Gamma^{\infty}(\End(E|_Y)) \bigr)^{K}.
\]
Lemma 3.5 in \cite{HW2} states that this heat kernel in fact lies in
\[
\bigl( \cC(G) \otimes \Gamma^{\infty}(\End(E|_Y)) \bigr)^{K},
\]
where $\cC(G)$ is Harish-Chandra's Schwartz algebra.  The operator $D_N$ decomposes into an operator on $G$ and one on $Y$, see e.g.\ (4.2) in \cite{HW2}. The operator on $G$ in this decomposition is defined in terms of the left (or right) regular representation by $\kg$ on $C^{\infty}(G)$, which preserves $\cC(G)$. As pointed out  in (3.7) in \cite{Moscovici89}, this implies that also
\[
\tilde \kappa_t \in \bigl( \cC(G) \otimes \Gamma^{\infty}(\End(E|_Y)) \bigr)^{K}.
\]
The trace $\Tr_g$ corresponds to the trace $\tau_g$ on $\cC(G)$ under this identification, combined with a trace on $\Gamma^{\infty}(\End(E|_Y))$ which converges by compactness of $Y$, see Lemma 3.1(c) in \cite{HW2}. (This is a special case of \eqref{eq Trg taug TR}.) Hence the claim follows.
\end{proof}

Lafforgue \cite{Lafforgue02} constructed an algebra 
satisfying the conditions of Proposition \ref{prop proj lim}.
 In fact, this algebra is a Banach algebra itself, rather than just a projective limit of Banach algebras. Furthermore, it contains heat kernels of Dirac operators. It is shown how to generalise the construction to reductive Lie groups in Section 4.2 of \cite{Lafforgue02}, but we will work with linear semisimple groups for simplicity.

Consider the spherical function $\varphi := \varphi_0^G$, as in (7.41) in \cite{Knapp}. Let $\kg = \kk \oplus \kp$ be a Cartan decomposition for a maximal compact subgroup $K<G$, and let $\ka \subset \kp$ be a maximal abelian subspace. Fix a positive restricted root system $\Sigma^+$ for $(\kg, \ka)$, and write
\[
\rho:= \sum_{\lambda \in \Sigma^+} \dim(\kg_{\lambda}) \lambda.
\]
Define the function $d_G\colon G \to [0, \infty)$ by
\[
d_G(k\exp(H)k') = \langle \rho, H\rangle,
\]
for $k,k' \in K$, and $H \in \ka$ positive with respect to $\Sigma^+$. Then $d_G(x)$ is the distance from $xK$ to $eK$ in $G/K$. For $t \in \R$, let $\cS_t(G)$ be the completion of $C_c(G)$ in the norm
\[
\|f\|_{t} := \sup_{g \in G} |f(g)| \varphi(g)^{-1} (1+d_G(g))^t.
\]
By Proposition 4.1.2 in \cite{Lafforgue02}, the algebra $\cS_t(G)$ is a Banach algebra, dense in $C^*_r(G)$ and closed under holomorphic functional calculus, for large enough $t$.
\begin{lemma} \label{lem taug cts ss}
For every semisimple element $g \in G$, and for $t$ large enough, the trace $\tau_g$ defines a continuous functional on $\cS_t(G)$.
\end{lemma}
\begin{proof}
Let $g \in G$ be semisimple.
By Theorem 6 in \cite{HCDSII}, the integral
\[
\int_{G/{\Zg}} (1+d_G(xgx^{-1}))^{-t} \varphi(xgx^{-1})d(x{\Zg})
\]
converges for $t$ large enough. So, for such $t$, and $f \in \cS_t(G)$,
\[
|\tau_g(f)| \leq \|f\|_{t}  \int_{G/{\Zg}} (1+d_G(xgx^{-1}))^{-t} \varphi(xgx^{-1})d(x{\Zg}).
\]
\end{proof}

\begin{lemma} \label{lem ss het kernel St}
For all $s >0$, the Schwartz kernels $\kappa_s$ and $\tilde \kappa_s$
lie in $\End_{\cS_t(G)}(\Gamma^{\infty}(E|_N))^G$.
\end{lemma}
\begin{proof}
As in the proof of Lemma \ref{lem ss g trace cl}, we see that these kernels lie in
\[
\bigl( \cC(G) \otimes \Gamma^{\infty}(\End(E|_Y))\bigr)^{K},
\]
where $\cC(G)$ is Harish-Chandra's Schwartz algebra, and we write $N = G\times_K Y$ as in Abels' theorem \cite{Abels}. The claim follows from the fact that $\cC(G) \subset \cS_t(G)$ for $t\geq0$.
\end{proof}

Propositions \ref{prop proj lim} and \ref{prop proj lim small t} and the lemmas in this subsection have the following consequence.
\begin{proposition}\label{prop deloc eta ss}
%
Suppose that $G$ is a real, linear, connected semisimple Lie group and  $g \in G$ is semisimple.
If $D_N$ is invertible, then $\eta_g(D_N)$ converges for large $t$. If, furthermore, 
 either $g$ has no fixed points or $D$ is a twisted $\Spinc$-Dirac operator, 
then $\eta_g(D_N)$ also converges for small $t$, so $D_N$ has a $g$-delocalised $\eta$-invariant.
\end{proposition}

 \section{Proof of the index theorem} \label{sec proof}

After discussing convergence of delocalised $\eta$-invariants in Section \ref{sec conv eta}, 
we are ready to prove Theorem \ref{thm APS} and Corollaries \ref{cor e}, \ref{cor APS discr} and \ref{cor APS ss}. We do this by 
defining a suitable parametrix of $D$, depending on a positive parameter $t$, and 
splitting up the right hand side of  \eqref{eq def g index}
  as as sum of several terms, and computing the limits of these terms as $t\downarrow 0$. These terms involve different parametrices of the Dirac operators on $\hat M$, $\tilde M$ and $C$.

 \subsection{Parametrices} \label{sec param}

 Consider the setting of Subsection \ref{sec bdry}.
Let $\psi_1\colon (0,\infty) \to [0,1]$ be a smooth function such that $\psi_1$ equals $1$ on $(0, \varepsilon)$ and $0$ on $(1-\varepsilon,\infty)$, for some $\varepsilon \in (0,1/2)$. Set $\psi_2 := 1-\psi_1$. Let $\varphi_1, \varphi_2\colon(0,\infty) \to [0,1]$ be smooth functions such that $\varphi_1$ equals $1$ on $(0,1-\varepsilon/2)$ and $0$ on $(1,\infty)$, while $\varphi_2$ equals $0$ on $(0,\varepsilon/4)$ and $1$  on $(\varepsilon/2,\infty)$.
Then $\varphi_j \psi_j = \psi_j$ for $j = 1,2$, and $\varphi_j'$ and $\psi_j$ have disjoint supports.

We pull back the functions $\varphi_j$ and $\psi_j$ to $C$ along the projection onto $(0,\infty)$, and extend these functions smoothly to $\hat M$ by setting $\psi_1$ and $\varphi_1$ equal to $1$ on $M\setminus U$,  and $\psi_2$ and $\varphi_2$ equal to $0$ on $M\setminus U$. We denote the resulting functions by the same symbols $\psi_j$ and $\varphi_j$. (No confusion is possible in what follows, because we will always use these symbols to denote the functions on $\hat M$.) We denote the derivatives of these functions in the $(0, \infty)$ directions by $\varphi_j'$ and $\psi_j'$, respectively. These derivatives are only defined  on $N \times (0,\infty) \subset \hat M$, so they will only be used there.

As before, let $\tilde M$ be the double of $M$, $\tilde E$ and $\tilde D$ the extensions of $E$ and $D$ to $\tilde M$. Fix $t >0$, and consider the parametrix
\beq{eq def tilde Q}
\tilde Q := \frac{1 - e^{-t\tilde D_- \tilde D_+}}{\tilde D_- \tilde D_+} \tilde D_-
\eeq
of $\tilde D_+$. (The part without the last factor $\tilde D_-$ is formed via functional calculus, by an application of the function $x\mapsto \frac{1-e^{-tx}}{x}$ to $\tilde D_- \tilde D_+$; this does not require invertibility of $\tilde D_- \tilde D_+$.)
Set
\beq{eq tilde Sj}
\begin{split}
\tilde S_{0}&:= 1-\tilde Q \tilde D_+ = e^{-t\tilde D_- \tilde D_+};\\
\tilde S_{1} &:= 1-\tilde D_+ \tilde Q=  e^{-t\tilde D_+ \tilde D_-}.
\end{split}
\eeq

The extension of $D_C$ to the complete manifold $N \times \R$ is essentially self-adjoint, and its square is positive. Hence its self-adjoint closure is invertible.
Let $Q_C$ be 
the inverse of that  closure, restricted to odd-graded sections.
Set
\beq{eq def R R'}
\begin{split}
R &:= \varphi_1 \tilde Q \psi_1 + \varphi_2 Q_C \psi_2.\\
R' &:= \psi_1 \tilde Q \varphi_1 + \psi_2 Q_C \varphi_2.\\
\end{split}
\eeq
Note that the operator $\tilde Q$ is well-defined on the supports of $\varphi_1$ and $\psi_1$, and that $Q_C$ is well-defined on the supports of $\varphi_2$ and $\psi_2$.

A parametrix like $R$ was also defined at the bottom of page 55 of \cite{APS1}. In \cite{APS1}, it was only used to prove that $D$ is Fredholm on spaces of sections with the {\APS}-boundary conditions. In this paper, it plays a central role in the proof of Theorem \ref{thm APS} as well. This is possible because, contrary to \cite{APS1},  we first consider the case where $D_C$ is invertible, so that we can use its inverse in the definition of $R$.

Set
\beq{eq def Sj}
\begin{split}
S_0 &:= 1-R\hat D_+;\\
S_0' &:= 1-R'\hat D_+;\\
S_1 &:= 1-\hat D_+ R.
\end{split}
\eeq


\begin{lemma} \label{lem comp Sj}
We have
\beq{eq S0 S1}
\begin{split}
S_0 &= \varphi_1 \tilde S_{0} \psi_1 + \varphi_1 \tilde Q \sigma \psi_1' + \varphi_2 Q_C \sigma \psi_2';\\
S_0' &= \psi_1 \tilde S_{0} \varphi_1 + \psi_1 \tilde Q \sigma \varphi_1' + \psi_2 Q_C \sigma \varphi_2';\\
S_1 &= \varphi_1 \tilde S_{1} \psi_1 - \varphi_1' \sigma \tilde Q \psi_1 - \varphi_2' \sigma Q_C \psi_2.
\end{split}
\eeq
\end{lemma}
\begin{proof}
These equalities follow from straightforward computations.
\end{proof}

\begin{lemma} \label{lem Sj smoothing}
The operators $S_0$ and $S_1$ have smooth kernels.
\end{lemma}
\begin{proof}
Because $\supp(\psi_j)\cap\supp(\varphi_j')=\emptyset$ for $j=1,2$ and the Schwartz kernels of the pseudo-differential operators $\tilde Q, Q_C$ are smooth off the respective diagonals in $\tilde M$ and $C$, the operators
$S_0', S_1$ have smooth Schwartz kernels. Hence so does
\[
S_0 = S_0'+(R'S_1 - S_0'R)\hat D_+.
\]
%
%
%
\end{proof}

\subsection{The interior contribution} \label{sec interior}

\begin{lemma} \label{lem S0 prime S1}
If $e^{-t\tilde D^2}$ is $g$-trace class, then so are
the operators $S_0'$ and $S_1$ in \eqref{eq def Sj}. And the function $\psi_1$ can be chosen such that
\[
 \Tr_g(S_0') - \Tr_g(S_1) = \frac{1}{2}\bigl( \Tr_g(e^{-t\tilde D_- \tilde D_+}) - \Tr_g(e^{-t\tilde D_+ \tilde D_-}) \bigr).
\]
\end{lemma}
\begin{proof}
In Lemma \ref{lem comp Sj}, the terms involving $\tilde S_0$ and $\tilde S_1$ are $g$-trace class because $e^{-t\tilde D^2}$ is. The functions $\varphi_j'$ and $\psi_j$ have disjoint, $G$-invariant supports. This implies that the terms in Lemma \ref{lem comp Sj} involving those functions are $g$-trace class, with $g$-trace equal to zero. (Here we use the fact that the operators $\tilde Q$ and $Q_C$ are pseudo-differential operators, and hence have smooth kernels off the diagonals of $\tilde M$ and $C$, respectively.) We find that $S_0'$ and $S_1$ are $g$-trace class, and
\[
\Tr_g(S_0') - \Tr_g(S_1) = \Tr_g(\psi_1 \tilde S_0 \varphi_1) - \Tr_g(\varphi_1 \tilde S_1 \psi_1).
\]
By 
$G$-invariance of $\psi_1$ and $\varphi_1$, and the fact that $\psi_1\varphi_1 = \psi_1$, we have
\[
\begin{split}
\Tr_g(\psi_1 \tilde S_1 \varphi_1) &= \Tr_g(\tilde S_1 \psi_1);\\
\Tr_g(\varphi_1 \tilde S_0 \psi_1)&= \Tr_g(\tilde S_0 \psi_1).
\end{split}
\]
Choosing $\psi_1$ such that $1_{\tilde M} - \psi_1$ equals the mirror image of $\psi_1$ on the double $\tilde M$, we have
\[
\Tr_g(\tilde S_j) = \Tr_g(\tilde S_j \psi_1) + \Tr_g(\tilde S_j (1_{\tilde M} -\psi_1)) =  2\Tr_g(\tilde S_j \psi_1).
\]
Hence the claim follows.
\end{proof}

\subsection{Another parametrix on $C$}

Next, we show how the boundary contribution $-\frac{1}{2}\eta_g(D_N)$ appears in Theorem \ref{thm APS}; see Proposition \ref{prop:key}.
We start by defining another parametrix for $D_C$, and showing that, in a suitable sense, its difference with $\tilde Q$ restricted to $U$ has zero $g$-trace in the limit $t \downarrow 0$. This is Lemma \ref{lem1}.

Let $\tilde \kappa_t$ be the Schwartz kernel of $e^{-t \tilde D_- \tilde D_+}\tilde D_-$, and $\kappa_{C, t}$  the Schwartz kernel of $e^{-tD_{C, -} D_{C, +}}D_{C, -} $.
\begin{lemma} \label{lem asympt exp}
For any $\varphi, \psi \in C^{\infty}(\hat M)^G$ supported in $N \times (0,1)$,
\[
(\varphi \otimes \psi)(\tilde \kappa_t - \kappa_{C, t}) \sim 0,
\]
as an asymptotic expansion in t, with respect to the $C^k$-norm on compact subsets of $\hat M$, for all $k$.
\end{lemma}
\begin{proof}
The Schwartz kernels of the heat operators $ e^{-t \tilde D_- \tilde D_+}$ and $e^{-tD_{C, -} D_{C, +}}$ have asymptotic expansions in $t$, with respect to the $C^k$-norm on compact subsets of $\hat M$, in which the coefficient of $t^{j+1}$ is the solution of a differential equation in terms of the local geometry of the manifold and the Dirac operator in question, and the coefficient of $t^{j}$. See for example Theorem 2.26 in \cite{BGV}, or the proof of Theorem 7.15 in \cite{Roe98}. Because the operators $\tilde D$ and $D_C$ are equal on $N \times (0,1)$, their heat kernels therefore have the same asymptotic expansion there. Since these asymptotic expansions may be differentiated term-wise, and $ \tilde D \psi = D_C\psi $, the claim follows.
%
\end{proof}

Let
\beq{eq def QC'}
Q'_C:=\frac{1-e^{-tD_{C,-}D_{C,+}}}{D_{C,-}D_{C,+}}D_{C,-}
\eeq
be an alternative parametrix for $D_{C,+}$ on the cylindrical end $C$. This operator is defined in terms functional calculus applied to the self-adjoint closure of the extension of $D_C$ to $N \times \R$.

When we wish to emphasise the dependence of operators on $t$, we write $\tilde Q(t)$ for $\tilde Q$, etc.
\begin{lemma}
\label{lem1}
Suppose that, for all $\varphi, \psi \in C^{\infty}(N \times \R)$ with supports in $N \times (0,1)$,  the operators  $\tilde D e^{-t \tilde D^2}$ and $\varphi D_C e^{-tD_C^2} \psi$
 are $g$-trace class for all $t>0$, and that their Schwartz kernels localise at $M^g$ as $t \downarrow 0$ in the sense of \eqref{eq loc Mg}.
%
%
%
Then
the operator $(\varphi_1\tilde Q-\varphi_2 Q'_C)\sigma\psi_1'$ is of $g$-trace class for all $t>0$, and
\beq{eq V}
\lim_{t\downarrow 0}\Tr_g((\varphi_1\tilde Q(t)-\varphi_2 Q'_C(t))\sigma\psi_1')=0.
\eeq
\end{lemma}
\begin{proof}
Choose a function $\varphi \in C^{\infty}(\hat M)$ such that for $j=1,2$, $\varphi$ equals $1$ on the support of $\psi_j'$, and zero outside the support of $1-\varphi_j$. With $\varepsilon > 0$ as at the start of Subsection \ref{sec param}, this is the case if $\varphi$ equals $0$ outside $N \times (\varepsilon/2, 1-\varepsilon/2)$ and $1$ on $N \times (\varepsilon, 1-\varepsilon)$. Then $\varphi \varphi_j = \varphi$, and $\psi_1'$ and $1-\varphi$ have disjoint supports. As in the proof of Lemma \ref{lem S0 prime S1}, this implies that
\[
(\varphi_1\tilde Q-\varphi_2 Q'_C)\sigma\psi_1' = \varphi(\tilde Q - Q_C')\sigma\psi_1' + (1-\varphi)(\varphi_1\tilde Q-\varphi_2 Q'_C)\sigma\psi_1',
\]
where the second term is $g$-trace class, with $g$-trace zero.
As on page 369 of \cite{CM90}, we write
\[
\frac{1-e^{-tA^*A}}{A^*A} = -\int_0^t e^{-sA^*A}\, ds,
\]
for any operator $A$ such that $A^*A$ is self-adjoint.
Then
\[
\varphi(\tilde Q - Q_C')\sigma \psi_1' = -\int_0^t \varphi (e^{-s \tilde D_- \tilde D_+} \tilde D_- -  e^{-s(D_{C, -} D_{C, +})}D_{C, -})\sigma\psi_1'\, ds.
\]

By assumption, $ e^{-t \tilde D_- \tilde D_+}\tilde D_-$ and $\varphi e^{-tD_{C, -} D_{C, +}}  D_{C, -}\psi_1'$
are $g$-trace class. Let $V \subset \hat M$ be a ${\Zg}$-invariant neighbourhood of $U^g$.
We use 
Lemma \ref{lem Trg Tr}, and note that by $G$-invariance of $\varphi$,
\begin{multline} \label{eq int on hat M}
 \int_{\hat M} \int_0^t  \chi_g(m)^2 \varphi(g^{-1}m) \psi_1'(m) \tr(g(\tilde \kappa_s - \kappa_{C, s})(g^{-1}m, m)\sigma)\, ds\, dm= \\
   \int_{V} \int_0^t  \chi_g(m)^2  \psi_1'(m) \tr(g(\tilde \kappa_s - \kappa_{C, s})(g^{-1}m, m)\sigma)\, ds\, dm
\\+
  \int_{U\setminus V} \int_0^t \chi_g(m)^2  \psi_1'(m) \tr(g(\tilde \kappa_s - \kappa_{C, s})(g^{-1}m, m)\sigma)\, ds\, dm.
\end{multline}

Since $\overline{U}^g/{\Zg}$ is compact, we can and will choose $V$ such that  $\overline{V}/{\Zg}$ is compact as well. The intersection of the support of $\chi_g$ and $\overline{V}$ is then compact.
So Lemma \ref{lem asympt exp} implies that the first term on the right hand side of \eqref{eq int on hat M} is bounded by a constant times any power of $t$. The second term goes to zero as $t\downarrow 0$ by the localisation assumption on  $\tilde D e^{-t \tilde D^2}$ and $\varphi D_C e^{-tD_C^2} \psi$.
\end{proof}


\subsection{The boundary contribution}

\begin{definition}\label{def eta t}
If $\eta_g(D_N)$ converges for large $t$, then for $t>0$ we set
\[
\eta^t_g(D_N) := \frac{2}{\sqrt{\pi}}\int_{\sqrt{t}}^{\infty}\Tr_g(e^{-s^2D_N^2}D_N)ds.
\]
\end{definition}

\begin{lemma}
\label{lem2}
If $\eta_g(D_N)$ converges for large $t$,  then  for all $t>0$, 
the operator $\varphi_2e^{-tD_{C,-}D_{C,+}}Q_C\sigma \psi_2'$ is  $g$-trace class, and
\[
\Tr_g(\varphi_2e^{-tD_{C,-}D_{C,+}}Q_C\sigma \psi_2')=-\frac12\eta^t_g(D_N).
\]
\end{lemma}
\begin{proof}
By definition, $D_{C,+}=\sigma \bigl(-\frac{\partial}{\partial u}+D_N\bigr)$, so $D_{C,-}=\bigl(\frac{\partial}{\partial u}+D_N \bigr) \sigma^{-1}$.
Hence, using an integral representation of $e^{-tD_{C,-}D_{C,+}} (D_{C,-}D_{C,+})^{-1}$ analogous to the expression for $Q(D)$ on page 369 of \cite{CM90}, we have
\begin{align}
e^{-tD_{C,-}D_{C,+}}Q_C&=e^{-tD_{C,-}D_{C,+}}D_{C,+}^{-1}=-\int_{t}^{\infty} e^{-sD_{C,-}D_{C,+}}D_{C,-}\, ds \nonumber \\
&=-\int_t^{\infty}e^{-s(-\frac{\partial^2}{\partial u}+D_N^2)}\bigl(\frac{\partial}{\partial u}+D_N\bigr) \sigma^{-1}\, ds \nonumber \\
&=-\int_t^{\infty}e^{s\frac{\partial^2}{\partial u^2}}e^{-sD_N^2}D_N \sigma^{-1} \, ds
-\int_t^{\infty}e^{s\frac{\partial^2}{\partial u^2}}\frac{\partial}{\partial u}e^{-sD_N^2} \sigma^{-1}\, ds. \label{eq lem 5.5}
\end{align}
We denote the Schwartz kernel of an operator $T$ by $\kappa_T$.
The Schwartz kernel of $e^{s\frac{\partial^2}{\partial u^2}}$ is the heat kernel of the Laplacian $-\frac{d^2}{du^2}$ on the real line:
\[
\kappa_{e^{s\frac{\partial^2}{\partial u^2}}}(u, u')=\frac{1}{\sqrt{4\pi s}}e^{-\frac{|u-u'|^2}{4s}}.
\]
The Schwartz kernel of $e^{s\frac{\partial^2}{\partial u^2}}\frac{\partial}{\partial u}$ is the derivative of the above heat kernel in the first variable:
\[
\kappa_{e^{s\frac{\partial^2}{\partial u^2}}\frac{\partial}{\partial u}}(u, u')=\frac{-1}{\sqrt{4\pi s}}\frac{u-u'}{2s}e^{-\frac{|u-u'|^2}{4s}}.
\]
Thus,
\beq{eq heat ker line}
\int_0^{\infty}\varphi_2(u)\kappa_{e^{s\frac{\partial^2}{\partial u^2}}}(u, u)\psi_2'(u)\, du=\int_0^{\infty}\frac{1}{\sqrt{4\pi s}}\psi_2'(u)du=\frac{1}{\sqrt{4\pi s}},
\eeq
and
\beq{eq der heat ker line}
\int_0^{\infty}\varphi_2(u)\kappa_{e^{s\frac{\partial^2}{\partial u^2}}\frac{\partial}{\partial u}}(u, u)\psi_2'(u)\, du=\int_0^{\infty}\frac{-1}{\sqrt{4\pi s}}\frac{u-u}{2s}\psi_2'(u)\, du=0.
\eeq

We choose the cutoff function $\chi$ such that on $U = N \times (0,1]$, it is the pullback of a function $\chi_N$ on $N$. Let $dn$ be the Riemannian density on $N$, so that $dm = dn\, du$ on $U$.
Using 
\eqref{eq lem 5.5}, \eqref{eq heat ker line} and \eqref{eq der heat ker line}, we then find that
\begin{multline*}
\Tr_g(\varphi_2e^{-tD_{C,-}D_{C,+}}Q_C \sigma \psi_2')=\\
-\int_t^{\infty}\left[\int_{G/{\Zg}}\int_N \chi_N(hgh^{-1}n)^2\tr \bigl(hgh^{-1}\kappa_{e^{-sD_N^2}D_N}(hg^{-1}h^{-1}n, n) \bigr)\, dn\, d(h{\Zg})\right] \\
\times \left[\int_0^{\infty}\varphi_2(u)\kappa_{e^{s\frac{\partial^2}{\partial u^2}}}(u, u)\psi_2'(u)\, du \right]\, ds\\
-\int_t^{\infty}\left[\int_{G/{\Zg}}\int_N \chi_N(hgh^{-1}n)^2\tr \bigl(hgh^{-1} \kappa_{e^{-sD_N^2}}(hg^{-1}h^{-1}n, n)\bigr) \, dn\, d(h{\Zg})\right] \\
\times \left[\int_0^{\infty}\varphi_2(u)\kappa_{e^{s\frac{\partial^2}{\partial u^2}}\frac{\partial}{\partial u}}(u, u)\psi_2'(u)\, du \right]\, ds\\
=\frac{-1}{\sqrt{4\pi}}\int_t^{\infty}\Tr_g(e^{-sD_N^2}D_N)\frac{1}{\sqrt{s}}\, ds
=\frac{-1}{\sqrt{\pi}}\int_{\sqrt{t}}^{\infty}\Tr_g(e^{-s^2D_N^2}D_N)\, ds.
\end{multline*}
In particular, this computation, with absolute values in the appropriate places,   shows that $\varphi_2e^{-tD_{C,-}D_{C,+}}Q_C \sigma \psi_2'$ is $g$-trace class if $\eta_g(D_N)$ converges for large $t$.
\end{proof}

%
%
%
%

The following proposition is the key step in the proof of Theorem \ref{thm APS}.
\begin{proposition}
\label{prop:key}
Suppose that $e^{-t\tilde D^2}$ is $g$-trace class for all $t>0$,  $\eta_g(D_N)$ converges for  large $t$,  $\tilde D e^{-t \tilde D^2}$ and $\varphi D_C e^{-tD_C^2} \psi$
 are $g$-trace class for all $t>0$ and all $\varphi, \psi \in C^{\infty}(C)^G$ supported in $N \times (0,1)$, and that their Schwartz kernels localise at $M^g$ as $t \downarrow 0$.
Then
 $S_0(t)$ is $g$-trace class for all $t>0$, and
\[
\Tr_g S_0(t)-\Tr_g S_0'(t) =-\frac12\eta^t_g(D_N) + F(t),
\]
for a function $F$ satisfying $\lim_{t \downarrow 0}F(t) = 0$.
\end{proposition}
\begin{proof}
By Lemma \ref{lem comp Sj},
\beq{eq S0 prime S0}
S_0 - S_0' =
 (\varphi_1 \tilde S_{0} \psi_1 - \psi_1 \tilde S_{0} \varphi_1)- ( \psi_1 \tilde Q \sigma \varphi_1' + \psi_2 Q_C \sigma \varphi_2') + (\varphi_1 \tilde Q \sigma \psi_1' + \varphi_2 Q_C \sigma \psi_2').
\eeq
The first term in brackets on the right hand side is $g$-trace class because $e^{-t\tilde D^2}$ is, and its $g$-trace is zero by Lemma \ref{lem Trg trace}.
As in the first paragraph of the proof of Lemma \ref{lem S0 prime S1}, the second term in brackets on the right hand side of \eqref{eq S0 prime S0} is $g$-trace class, with $g$-trace zero.

Using the fact that $\psi_1'+\psi_2' = 0$, we rewrite the third term in brackets on the right hand side of \eqref{eq S0 prime S0} as
\beq{eq third term}
\varphi_1 \tilde Q \sigma \psi_1'+\varphi_2Q_C \sigma \psi_2' = (\varphi_1 \tilde Q - \varphi_2 Q_C')\sigma \psi_1' + \varphi_2(Q_C - Q_C')\sigma \psi_2'.
\eeq
Since $Q_C$ is the exact inverse of $D_{C,+}$,
\[
 Q_C-Q'_C
=Q_C-\frac{1-e^{-tD_{C,-}D_{C,+}}}{D_{C,-}D_{C,+}}D_{C,-}(D_{C,+}Q_C) 
=e^{-tD_{C,-}D_{C,+}}Q_C.
\]
So Lemmas \ref{lem1} and \ref{lem2} imply that \eqref{eq third term} is $g$-trace class,
and that its $g$-trace equals
\[
\frac12\eta^t_g(D_N) + \Tr_g((\varphi_1\tilde Q(t)-\varphi_2 Q'_C(t))\sigma\psi_1'),
\]
where the second term on the right goes to zero as $t \downarrow 0$.

Because $S_0'$ is $g$-trace class by Lemma \ref{lem S0 prime S1}, we find that $S_0$ is also $g$-trace class, which completes the proof.
\end{proof}

\subsection{Proofs of index theorems} \label{sec pf APS}

\begin{proof}[Proof of Theorem \ref{thm APS}]
Under the conditions in Theorem \ref{thm APS}, Lemma \ref{lem S0 prime S1} and Proposition \ref{prop:key} imply that $S_0 = S_0(t)$ and $S_1 = S_0(t)$ are $g$-trace class for all $t>0$, so $D$ is $g$-Fredholm.
%
By Lemma \ref{lem S0 prime S1} and Proposition \ref{prop:key}, we have for all $t>0$, 
\beq{eq pf APS}
\ind_g(\hat D) = 
\Tr_g(S_0(t))-\Tr_g(S_1(t)) = \frac{1}{2} (\Tr_g(e^{-t\tilde D_- \tilde D_+}) - \Tr_g(e^{-t\tilde D_+ \tilde D_-})) - \frac{1}{2}\eta_g^t(D_N) + F(t),
\eeq
where $\lim_{t \downarrow 0} F(t) = 0$. The first term on the right is independent of $t$, and equals $\frac{1}{2} \ind_g(\tilde D)$. The left hand side of \eqref{eq pf APS} is also independent of $t$. We find that $\lim_{t\downarrow 0} \eta_g^t(D_N) $ converges, i.e.\ $\eta_g(D_N)$ converges for small $t$. And \eqref{eq APS} holds in the limit $t \downarrow 0$.
\end{proof}


We prove Corollaries \ref{cor e}, \ref{cor APS discr} and \ref{cor APS ss} by verifying that the conditions of Theorem \ref{thm APS} hold in those settings, and applying existing index theorems for manifolds without boundary.

If $g=e$, then all conditions in
Theorem \ref{thm APS} hold. The topological expression for the interior contribution is Theorem 6.10 or Theorem 6.12 in \cite{Wang14},
so Corollary \ref{cor e} follows.

To deduce Corollaries  \ref{cor APS discr} and \ref{cor APS ss} from Theorem \ref{thm APS}, we first note that
 if $\alpha$ is a volume form on (each connected component of) $M^g$, and $\tilde \alpha$ is its extension to $\tilde M^g$,  then
\beq{eq int tilde Mg}
\int_{M^g} \alpha = \frac{1}{2}\int_{\tilde M^g} \tilde \alpha,
\eeq
if either of the integrals converges.
This is because near $N$, the fixed point set $\tilde M^g$ equals the Cartesian product of $N^g$ and an interval. So there are no components of $\tilde M^g$ contained in $N$, so that there is no `double counting' in passing from $\tilde M^g$ to two copies of $M^g$. Because both the orientation on $TM$ and the grading on $E$ are reversed on the second copy of $M$ in $\tilde M$, the integrand on the right hand side of \eqref{eq index formula} is the extension of the corresponding integrand on $M^g$ in this sense,
so \eqref{eq int tilde Mg} applies in that case.

Let us check that the localisation property \eqref{eq loc Mg} of  $e^{-t\tilde D^2} \tilde D$ and $\varphi e^{-tD_C^2} D_C \psi$ assumed in Theorem \ref{thm APS} hold in the settings of Corollaries \ref{cor APS discr} and \ref{cor APS ss}.
\begin{proposition} \label{prop loc discr}
Suppose $\varphi, \psi \in C^{\infty}(C)$ are $G$-invariant and supported in $U$.
If $G$ is a finitely generated, discrete group, then the kernels of  $e^{-t\tilde D^2} \tilde D$ and $\varphi e^{-tD_C^2} D_C \psi$ localise at fixed point sets, as $t \downarrow 0$.
\end{proposition}
\begin{proof}
Let $\tilde \kappa_t$ be the kernel of $e^{-t\tilde D^2}\tilde D$ and let $V$ be a neighbourhood of $\tilde M^g$ in $\tilde M$ where $\bar V$ is ${\Zg}$-compact.
Then, with $\chi_g$ as in \eqref{eq tilde chi},
\begin{align*}
&\int_{\tilde M\backslash V}\chi_g(m)^2|\tr(g\tilde \kappa_t(g^{-1}m, m))|\, dm\\
=&\sum_{h\in(g)}\int_{\tilde M\backslash V}\chi(m)^2|\tr(h\tilde \kappa_t(h^{-1}m,m))|\, dm\\
\le &\|\chi\|_{\infty} \sum_{h\in(g)}\int_{(\tilde M\backslash V)\cap \supp(\chi)}|\tr(h\tilde \kappa_t(h^{-1}m,m))|\, dm.
\end{align*}
The last sum of integrals converges because of the compactness of $(\tilde M\backslash V)\cap\supp (\chi)$ and the uniform convergence of $\sum_{h\in G}|\tr(h\tilde \kappa_t(h^{-1}m, m))|$. Here we use Corollary 3.5 from \cite{Wangwang}, as in Lemma \ref{lem discr g trace cl}.
This ensures that the limit as $t$ goes to $0$ can be evaluated inside the sum and the integral.
Therefore
\[
\lim_{t\downarrow 0}\int_{\tilde M\backslash V}\chi_g(m)^2|\tr(g\tilde \kappa_t(g^{-1}m, m))|\, dm=0,
\]
because if $m\in\tilde M\backslash V$ and $h\in (g)$, then $\tr(h\tilde \kappa_t(h^{-1}m, m))\to 0$ as $t\downarrow 0$.

In a similar way, one  can prove that the operator $\varphi e^{-tD_C^2} D_C \psi$ localises at fixed point sets, noting that $U\cap \supp(\chi)$ is relatively compact.
\end{proof}

\begin{proposition} \label{prop loc ss}
Suppose $\varphi, \psi \in C^{\infty}(C)$ are $G$-invariant and supported in $U$.
If $G$ is a real, linear, connected semisimple Lie group, then the kernels of  $e^{-t\tilde D^2} \tilde D$ and $\varphi e^{-tD_C^2} D_C \psi$ localise at fixed point sets, as $t \downarrow 0$.
\end{proposition}
\begin{proof}
The claim follows from cocompactness of $\tilde M$ and the supports of $\varphi$ and $\psi$, by Proposition 4.6 in \cite{HW2}. That result applies to the Schwartz kernels of $e^{-t\tilde D^2}$ and $\varphi e^{-tD_C^2} \psi$; this extends to $e^{-t\tilde D^2} \tilde D$ and $\varphi e^{-tD_C^2} D_C \psi$ because the kernels those operators have Gaussian off-diagonal decay behaviour as well.
See Section 4.3 of \cite{HW2}, with the
 reference  to  Theorem 4 in \cite{CLY81} replaced by a reference to Theorem 6 in \cite{CLY81}. The only difference is an extra factor $t^{-1}$ in the function bounding the kernel, and the arguments can easily be modified to incorporate this. Furthermore, the condition in \cite{HW2} that group elements have ``finite Gaussian orbital integrals" is shown to hold for all semisimple elements in Section 4.2 of \cite{BismutHypo}.
\end{proof}

\begin{proof}[Proof of Corollary \ref{cor APS discr}.]
In the setting of Corollary \ref{cor APS discr},
Proposition \ref{prop deloc eta discr}  implies that $D_N$ has a $g$-delocalised $\eta$-invariant.
Entirely analogously to Lemma \ref{lem discr g trace cl}, one finds that $e^{-t\tilde D^2}$ and $e^{-t\tilde D^2} \tilde D$ are $g$-trace class. 
%
%
The localisation property of  $e^{-t\tilde D^2} \tilde D$ and $\varphi e^{-tD_C^2} D_C \psi$ (for $G$-invariant $\varphi, \psi \in C^{\infty}(C)$, supported in $U$) is Proposition \ref{prop loc discr}.
The index formula \eqref{eq index formula} on manifolds without boundary holds by Theorem 5.10 in \cite{Wangwang}. By \eqref{eq int tilde Mg} and the comments below it, the right hand side of \eqref{eq index formula} can be rewritten as an integral over $M^g$. Hence Theorem \ref{thm APS} implies Corollary \ref{cor APS discr}.
\end{proof}

\begin{proof}[Proof of Corollary \ref{cor APS ss}.]
In the setting of Corollary \ref{cor APS ss}, 
Proposition  \ref{prop deloc eta ss} implies that $D_N$ has a $g$-delocalised $\eta$-invariant.
Similarly to the discrete case, arguments like those in the proof of  Lemma \ref{lem ss g trace cl}
 imply that $e^{-t\tilde D^2}$ $e^{-t\tilde D^2} \tilde D$ are $g$-trace class. 
The localisation property of  $e^{-t\tilde D^2} \tilde D$ and $\varphi e^{-tD_C^2} D_C \psi$ (for $G$-invariant $\varphi, \psi \in C^{\infty}(C)$, supported in $U$) is Proposition \ref{prop loc ss}.
The index formula \eqref{eq index formula} now holds by Theorem 2.5 in \cite{HW2}, and again we use  \eqref{eq int tilde Mg} to rewrite the right hand side of \eqref{eq index formula}  as an integral over $M^g$. In the application of Theorem 2.5 in \cite{HW2}, we use the fact that the condition
 that group elements have ``finite Gaussian orbital integrals" holds for all semisimple elements, as shown in Section 4.2 of \cite{BismutHypo}.
We conclude that Corollary \ref{cor APS ss} follows from Theorem \ref{thm APS}.
\end{proof}

\section{Non-invertible $D_N$} \label{sec DN not inv}

We have so far assumed the operator $D_N$ to be invertible. 
We now weaken this assumption, and only assume that $0$ is isolated in the spectrum of $D_N$. We indicate how to modify the arguments in the preceding sections to obtain  generalisations of Theorem \ref{thm APS} and Corollaries \ref{cor e}, \ref{cor APS discr} and \ref{cor APS ss} to that case.

\subsection{An index theorem for non-invertible $D_N$}

Let $\varepsilon > 0$ be such that $([-2\varepsilon, 2\varepsilon] \cap \spec (D_N))\setminus \{0\} = \emptyset$. Let $\psi \in C^{\infty}(\hat M)^G$ be a nonnegative function such that
\[
\begin{split}
\psi(n,u) &= \left\{
\begin{array}{ll}
u & \text{if $n \in N$ and $u \in (1/2, \infty)$;} \\
0 & \text{if $n \in N$ and $u \in (0,1/4)$;}
\end{array}
\right. \\
\psi(m) &= 0 \quad \text{if $m \in M \setminus U$}.
\end{split}
\]
(Recall that $U \cong N \times (0,1]$ is a neighbourhood of $N$ in $M$.)

Consider the operator
\[
\hat D_{\varepsilon} := e^{\varepsilon \psi} \hat D e^{-\varepsilon \psi}.
\]
The operator $\hat D_{\varepsilon}$ is $G$-equivariant, odd-graded and elliptic.
It is symmetric on $\Gamma_c^{\infty}(E)$, because it equals $D$ on $M \setminus (N \times (0,1/4))$, and
\beq{eq hat D epsilon}
\sigma \Bigl(-\frac{\partial}{\partial u} + D_N + \varepsilon \psi'\Bigr)
\eeq
on $C$. It is essentially self-adjoint by a finite propagation speed argument; see e.g.\ Proposition 10.2.11 in \cite{Higson00}.
On $\hat M \setminus M$, the operator \eqref{eq hat D epsilon} equals
\beq{eq DC epsilon}
\sigma \Bigl(-\frac{\partial}{\partial u} + D_N + \varepsilon \Bigr).
\eeq
The point is that the operator $D_N + \varepsilon$ is invertible, and many arguments that apply in the case where $D_N$ itself is invertible now apply with $D_N$ replaced by $D_N + \varepsilon$.

A difference between $\hat D$ and $\hat D_{\varepsilon}$, and between $D_N$ and $D_N + \varepsilon$,  is that the operators  $\hat D_{\varepsilon}$ and  $D_N + \varepsilon$ are not Dirac operators of the form \eqref{eq Dirac op}. One implication is that $\eta_g(D_N+ \varepsilon)$ may not converge for small $t$. To circumvent this, we use a standard  regularisation technique. We note that $\eta_g(D_N+ \varepsilon)$ does converge for large $t$, as a version of Proposition \ref{prop proj lim} still applies. Hence for all $t>0$, the number $\eta_g^t(D_N)$ as in  Definition \ref{def eta t} is well-defined.
\begin{definition}
For a function $f(t)$ that has an asymptotic expansion in $t$ as $t \downarrow 0$, the \emph{regularised limit} $\LIM_{t \downarrow 0}f(t)$ is
the coefficient of $t^0$ in such an asymptotic expansion.

The \emph{regularised $g$-delocalised $\eta$-invariant} of $D_N+ \varepsilon$ is
\[
\eta_g^{\reg}(D_N + \varepsilon) := \LIM_{t \downarrow 0} \eta_g^t(D_N+ \varepsilon).
\]
\end{definition}


Let $P \colon L^2(E|_N) \to \ker(D_N)$ be the orthogonal projection.
Theorem \ref{thm APS} generalises as follows.
\begin{theorem} \label{thm DN not inv}
Suppose that
\begin{itemize}
\item the operators $e^{-t\tilde D^2}$ and $e^{-t\tilde D^2} \tilde D$  are  $g$-trace class for all $t>0$;
\item the Schwartz kernels of $e^{-t\tilde D^2}\tilde D$ and $\varphi e^{-tD_C^2}D_C \psi$ localise at $M^g$ as $t \downarrow 0$, for all $G$-invariant $\varphi, \psi \in C^{\infty}(C)$, supported in $U$;
\item  $\eta_g(D_N+\varepsilon)$ converges for large $t$;
\item the projection $P$ is $g$-trace class.
\end{itemize}
Then  $\hat D_{\varepsilon}$ is $g$-Fredholm, and 
\beq{eq DN not inv}
\ind_g(\hat D_{\varepsilon}) = \frac{1}{2}  \LIM_{t\downarrow 0}
\bigl(\Tr_g(e^{-t\tilde D_{\varepsilon, -} \tilde D_{\varepsilon, +}}) - \Tr_g(e^{-t\tilde D_{\varepsilon, +} \tilde D_{\varepsilon, -}}) \bigr) - \frac{1}{2}\eta_g^{\reg}(D_N+ \varepsilon).
\eeq
\end{theorem}

\begin{corollary} \label{cor DN not inv}
Suppose $D$ is a $\Spinc$-Dirac operator twisted by a vector bundle $V$,
and that either
\begin{itemize}
\item $G$ is any unimodular locally compact group, and $g=e$;
\item $G$ is discrete and finitely generated, and the conjugacy class of $g$ has polynomial growth; or
\item $G$ is a connected, linear, real semisimple Lie group, and $g$ is semisimple.
\end{itemize}
Then  $\hat D_{\varepsilon}$ is $g$-Fredholm, $D_N$ has a $g$-delocalised $\eta$-invariant,  and 
%
\beq{eq cor DN not inv}
\lim_{\varepsilon \downarrow 0}
\ind_g(\hat D_{\varepsilon})  = \frac{1}{2}\int_{\tilde M^g} \chi_g^2 \frac{\hat A(\tilde M^g) e^{c_1(L|_{\tilde M^g})/2} \tr(ge^{-R_{\tilde V}|_{\tilde M^g}/2\pi i}) }{\det(1-g e^{-R_{\tilde \cN}/2\pi i})^{1/2}} - \frac{\Tr_g(P)+\eta_g(D_N)}{2}.
\eeq
\end{corollary}

\begin{remark} \label{rem DN not inv cpt}
The operator $\hat D_{\varepsilon}$ is unitarily equivalent to the operator $\hat D$ acting on Sobolev spaces weighted by the function $e^{\varepsilon \psi}$. This point of view is related to the $b$-calculus approach to the {\APS} index theorem \cite{Melrose}.
In the setting of \cite{APS1}, where $M$ is compact and $G$ is trivial, the index of $\hat D$ on those weighted Sobolev spaces equals the dimension of the kernel of $D_+$ on $L^2$-sections minus the dimension of the kernel of $D_-$ on \emph{extended} $L^2$-sections. See Section 6.5 in \cite{Melrose}. By Proposition 3.11 in \cite{APS1}, that number equals the {\APS} index of $D_+$ in the classical sense. Furthermore, all conditions in Theorem \ref{thm DN not inv} trivially hold in the compact case. On the right hand side of  \eqref{eq DN not inv}, 
the term $\frac{1}{2} \ind_e(\tilde D)$ equals the interior contribution in the \APS-index theorem, by Lemma \ref{lem fin dim} and the Atiyah--Singer index theorem. Hence Theorem \ref{thm DN not inv} reduces to the original {\APS} index theorem (Theorem 3.10 in \cite{APS1}) in this case.

In the compact case, $\ind_g(\hat D_{\varepsilon})$ is independent of $\varepsilon$ by a homotopy argument, so the limit $\lim_{\varepsilon \downarrow 0}$ may be omitted in \eqref{eq cor DN not inv}. It seems likely that this applies in certain noncompact situations as well, in particular in cases where $\ind_g(\hat D_{\varepsilon})$ can be recovered from a $K$-theoretic index, as in \cite{HWWII, PPST}.
%
 %
\end{remark}

\begin{remark}
\label{rem spec.sec}
In the setting of a Galois covering (when $G$ is discrete and acts freely), the perturbation of the boundary operator $D_N$ is equivalent to the choice of a spectral section in the sense of Leichtnam--Piazza \cite{Leichtnam98}.
The notion of a spectral section was initially introduced by Melrose-Piazza \cite{Melrose-Piazza} as an analogue of the {\APS} boundary conditions for a family of elliptic operators on a manifold with boundary and then generalised to the higher index associated to a Galois cover of a manifold with boundary \cite{Leichtnam98}.
A spectral section exits if and only if the higher index of the boundary Dirac operator vanishes, i.e., $\ind_{G}(D_N)=0\in K_*(C^*_r(G)).$
In our setting the bordism invariance of the index implies vanishing of the higher index of $D_N$, so a spectral section always exists.
Translating to the case of a manifold with a cylindrical end, the existence of a spectral section is equivalent to existence of  a perturbation of $D_N$ so that the perturbed operator ($\hat D_{\varepsilon}$ in our notation) is a generalised Fredholm operator.
For example, Wahl considered manifolds with cylindrical ends and proved an {\APS} index theorem for Dirac operator on Hilbert $C^*$-modules by performing a perturbation on the boundary Dirac operator which is equivalent to the choice of a boundary condition given by a spectral section~\cite{Wahl.AJM}.
\end{remark}

\subsection{Proof of Theorem \ref{thm DN not inv}}

Let $\tilde \psi$ be any smooth, $G$-invariant extension of $\psi|_M$ to the double $\tilde M$ of $M$.  Set
\[
\tilde D_{\varepsilon} := e^{\varepsilon \tilde \psi} \tilde D e^{-\varepsilon \tilde \psi},
\]
and
\[
\tilde Q_{\varepsilon} := \frac{1-e^{-t\tilde D_{\varepsilon, -}\tilde D_{\varepsilon, +} }}{\tilde D_{\varepsilon, -}\tilde D_{\varepsilon, +}}\tilde D_{\varepsilon, -}.
\]
Set
\[
\begin{split}
\tilde S_{\varepsilon, 0}&:= 1-\tilde Q_{\varepsilon} \tilde D_{\varepsilon, +} = e^{-t\tilde D_{\varepsilon, -}\tilde D_{\varepsilon, +} };\\
\tilde S_{\varepsilon, 1} &:= 1- \tilde D_{\varepsilon, +}  \tilde Q_{\varepsilon}=  e^{-t\tilde D_{\varepsilon, +}\tilde D_{\varepsilon, -}}.
\end{split}
\]

Let $D_{C, \varepsilon}$ be the restriction of $\hat D_{\varepsilon}$ to $N \times (1/2, \infty)$. This operator equals \eqref{eq DC epsilon}, and the self-adjoint closure of its extension to $N \times \R$ is invertible. Let $Q_{C, \varepsilon}$ be its inverse, restricted to sections of $E_-$.

Let the functions $\varphi_j$ and $\psi_j$ be as in Subsection \ref{sec param}, with the difference that they change values between $0$ and $1$ on the interval $(1/2, 1)$ rather than on $(0,1)$.
 Set
\[
R_{\varepsilon} := \varphi_1\tilde Q_{\varepsilon}  \psi_1 + \varphi_2 Q_{C, \varepsilon} \psi_2.
\]
Set
\[
\begin{split}
S_{\varepsilon, 0} &:= 1-R_{\varepsilon}\hat D_{\varepsilon, +};\\
S_{\varepsilon, 1} &:= 1-\hat D_{\varepsilon, +}R_{\varepsilon}.
\end{split}
\]

From now on, the proof of Theorem \ref{thm DN not inv} is analogous to the proof of Theorem \ref{thm APS}, apart from the fact that $\tilde D_{\varepsilon}$ and $D_N + \varepsilon$ are not Dirac operators of the form \eqref{eq Dirac op}.
As in Lemma \ref{lem comp Sj}, we have
\[
\begin{split}
S_{\varepsilon, 0} &= \varphi_1 \tilde S_{\varepsilon, 0} \psi_1 + \varphi_1 \tilde Q_{\varepsilon}\sigma  \psi_1' + \varphi_2 Q_{C, \varepsilon} \sigma \psi_2';\\
S_{\varepsilon, 1} &= \varphi_1 \tilde S_{\varepsilon, 1} \psi_1 - \varphi_1'\sigma  \tilde Q_{\varepsilon}  \psi_1 - \varphi_2' \sigma Q_{C, \varepsilon} \psi_2.\\
\end{split}
\]
In the same way as for $S_0$ and $S_1$, we find that $S_{\varepsilon, 0}$ and $S_{\varepsilon, 1}$
 have smooth kernels, and they are $g$-trace class. 
 And Proposition \ref{prop:key} generalises directly to $\hat D_{\varepsilon}$, and we obtain the following analogue of \eqref{eq pf APS}:
\beq{eq pf APS 2}
\ind_g(\hat D_{\varepsilon}) = 
\frac{1}{2} (\Tr_g(e^{-t\tilde D_{\varepsilon, -} \tilde D_{\varepsilon, +}}) - \Tr_g(e^{-t\tilde D_{\varepsilon, +} \tilde D_{\varepsilon, -}})) - \frac{1}{2}\eta_g^t(D_N+ \varepsilon) + F(t),
\eeq
for a function $F$ such that $\lim_{t \downarrow 0} F(t) = 0$. Here we used the fact that $\eta_g(D_N+\varepsilon)$ converges for large $t$ as Proposition \ref{prop proj lim} still applies. But
because  $\tilde D_{\varepsilon}$ and $D_N + \varepsilon$ are not Dirac operators of the form \eqref{eq Dirac op}, the first and second terms on the right hand side of \eqref{eq pf APS 2} may not have well-defined  limits as $t \downarrow 0$. But the left hand side is independent of $t$, so we obtain \eqref{eq DN not inv}.


\subsection{Delocalised $\eta$-invariants and projection onto $\ker D_N$}

Corollary \ref{cor DN not inv} follows from Theorem \ref{thm DN not inv} in the same way that Corollaries \ref{cor e},  \ref{cor APS discr} and  \ref{cor APS ss} follow from Theorem \ref{thm APS}, with three additional ingredients.

 The first is a generalisation of Proposition \ref{prop proj lim} to the case where $D_N$ is not invertible, but $0$ is isolated in its spectrum. That case can be deduced from the case where $D_N$ is invertible, see Remark \ref{rem conv eta non inv}. 
 
 The second additional ingredient is the following fact.
 \begin{lemma}\label{lem P g tr cl}
 In the setting of Corollary \ref{cor DN not inv}, 
the projection $P$ is $g$-trace class. 
 \end{lemma}
 \begin{proof}
 By the first part of the proof of Proposition 5.22 in \cite{PPST}, the operator $P$ has a smooth kernel in $\End_{\cA}(\Gamma^{\infty}(E|_N))^G$, for the relevant algebra $\cA$ as in Subsection \ref{sec proj lim}. Here we use that $\End_{\cA}(\Gamma^{\infty}(E|_N))^G$ is closed under holomorphic functional calculus. By \eqref{eq Trg taug TR}, this implies that $P$ is $g$-trace class.
 %
\end{proof}

The proof of the third additional ingredient is a modification of the proof of Proposition 5.14 in \cite{PPST}, which, in turn, is a variation on the proof of Proposition 8.38 in \cite{Melrose}. 
\begin{lemma} \label{lem eta DN eps}
 In the setting of Corollary \ref{cor DN not inv}, 
\[
\lim_{\varepsilon \downarrow 0}
\eta^{\reg}_g(D_N + \varepsilon) = \Tr_g(P) + \eta_g^{\reg}(D_N).
\]
\end{lemma}
\begin{proof}
We have
\beq{eq eta reg decomp}
\eta^{\reg}_g(D_N + \varepsilon) = \LIM_{t\downarrow 0} \frac{1}{\sqrt{\pi}} \int_{t}^1 \Tr_g(e^{-s(D_N+\varepsilon)^2} (D_N+ \varepsilon)) \frac{ds}{\sqrt{s}} + \lim_{T \to \infty}
\frac{1}{\sqrt{\pi}} \int_{1}^{T} \Tr_g(e^{-s(D_N+\varepsilon)^2} (D_N+ \varepsilon)) \frac{ds}{\sqrt{s}}. 
\eeq
To evaluate the first term on the right hand side in the limit $\varepsilon \downarrow 0$, we note that the Schwartz kernel of $e^{-s(D_N+\varepsilon)^2} (D_N+ \varepsilon)$ and the terms in its asymptotic expansion in $t$ are continuous in $\varepsilon$, uniformly in $s$ in compact intervals in $(0,\infty)$. Using this, 
we find that
\beq{eq eta reg small t}
\lim_{\varepsilon \downarrow 0} \LIM_{t\downarrow 0} \frac{1}{\sqrt{\pi}} \int_{t}^1 \Tr_g(e^{-s(D_N+\varepsilon)^2} (D_N+ \varepsilon)) \frac{ds}{\sqrt{s}} 
=\LIM_{t\downarrow 0} \frac{1}{\sqrt{\pi}} \int_{t}^1 \Tr_g(e^{-sD_N^2} D_N) \frac{ds}{\sqrt{s}}.
\eeq 

To evaluate the second term on the right hand side of \eqref{eq eta reg decomp}, we note that
\[
\begin{split}
\frac{d}{d\varepsilon} 
 \int_{1}^{T} \Tr_g(e^{-s(D_N+\varepsilon)^2} (D_N+ \varepsilon)) \frac{ds}{\sqrt{s}} &= 
  \int_{1}^{T} \Tr_g\bigl(e^{-s(D_N+\varepsilon)^2} \bigl(1-2s(D_N+ \varepsilon)^2\bigr) \bigr) \frac{ds}{\sqrt{s}} \\
  &=   \int_{1}^{T} \frac{d}{ds} \left( 
  2\sqrt{s} 
  \Tr_g(e^{-s(D_N+\varepsilon)^2}) 
  \right) \, ds\\
  &=  2\sqrt{T} 
  \Tr_g(e^{-T(D_N+\varepsilon)^2}) -  2
  \Tr_g(e^{-(D_N+\varepsilon)^2}). 
\end{split}
\]
We find that
\begin{multline}\label{eq eta reg large T}
 \int_{1}^{T} \Tr_g(e^{-s(D_N+\varepsilon)^2} (D_N+ \varepsilon)) \frac{ds}{\sqrt{s}} = \\
  \int_{1}^{T} \Tr_g(e^{-sD_N^2} D_N) \frac{ds}{\sqrt{s}} + 
2 \sqrt{T}  \int_0^{\varepsilon}   \Tr_g(e^{-T(D_N+\varepsilon')^2})\, d\varepsilon'- 
2   \int_0^{\varepsilon}   \Tr_g(e^{-(D_N+\varepsilon')^2})\, d\varepsilon'.
\end{multline}
The last term on the right hand side is independent of $T$, and the integrand is bounded in $\varepsilon'$. So
\beq{eq eta reg large T 3}
\lim_{\varepsilon \downarrow 0} \lim_{T \to \infty} 2   \int_0^{\varepsilon}   \Tr_g(e^{-(D_N+\varepsilon')^2})\, d\varepsilon' = 0.
\eeq

In the second term on the right hand side of \eqref{eq eta reg large T}, we note that
\beq{eq decomp DN heat}
e^{-T(D_N+\varepsilon')^2} = e^{-T(D_N^2+2\varepsilon' D_N)} e^{-T\varepsilon'^2}.
\eeq
And for all $\varepsilon' \in (0, \varepsilon)$, zero is isolated in the spectrum of the operator $D_N^2+2\varepsilon' D_N$, and this operator has the same kernel as $D_N$. Set 
\beq{eq RT}
R_T:=
e^{-T(D_N^2+2\varepsilon' D_N)}  - P.
\eeq
This is the action of the heat operator of the Laplace-type operator $D_N^2+2\varepsilon' D_N$ on the complement of its kernel.
The operator  on left hand side of \eqref{eq decomp DN heat} is $g$-trace class, hence so is the first operator on the right hand side of \eqref{eq RT}. The operator $P$ is $g$-trace class by Lemma \ref{lem P g tr cl}. Hence $R_T$ is $g$-trace class. By the arguments in the proof of Lemma \ref{lem rapid decr}, the $g$-trace of $R_T$ goes to zero as a Gaussian function as $T \to \infty.$


Because of this decay behaviour, and the fact that $P$ is $g$-trace class, 
\[
\lim_{T \to \infty} 2 \sqrt{T}  \int_0^{\varepsilon}   \Tr_g(e^{-T(D_N+\varepsilon')^2})\, d\varepsilon' = 
 \lim_{T \to \infty} 2\sqrt{T}  \int_0^{\varepsilon}  e^{-T\varepsilon'^2} \, d\varepsilon' \Tr_g(P).
\]
And for all $\varepsilon>0$, 
\[
 \lim_{T \to \infty} 2 \sqrt{T}  \int_0^{\varepsilon}  e^{-T\varepsilon'^2} \, d\varepsilon' = \int_0^{\infty} e^{-\varepsilon''^2}\, d\varepsilon'' = \sqrt{\pi}.
\]
So the second term on the right hand side of \eqref{eq eta reg large T} satisfies
\beq{eq eta reg large T 2}
\lim_{\varepsilon \downarrow 0} \lim_{T \to \infty} 2 \sqrt{T}  \int_0^{\varepsilon}   \Tr_g(e^{-T(D_N+\varepsilon')^2})\, d\varepsilon' = \sqrt{\pi} \Tr_g(P).
\eeq
Using \eqref{eq eta reg large T} and the expressions \eqref{eq eta reg large T 3} and \eqref{eq eta reg large T 2} for the relevant limits of the second and third terms on the right, and the fact that that $\eta_g(D_N)$ converges for large $t$ by the generalisation of Proposition \ref{prop proj lim} mentioned at the start of this subsection,
 we conclude that
\[
\lim_{\varepsilon \downarrow 0} \lim_{T \to \infty}
 \int_{1}^{T} \Tr_g(e^{-s(D_N+\varepsilon)^2} (D_N+ \varepsilon)) \frac{ds}{\sqrt{s}} = 
  \int_{1}^{\infty} \Tr_g(e^{-sD_N^2} D_N) \frac{ds}{\sqrt{s}} + \sqrt{\pi} \Tr_g(P).
\]
Together with \eqref{eq eta reg decomp} and \eqref{eq eta reg small t}, this implies the claim.
\end{proof}

\begin{proof}[Proof of Corollary \ref{cor DN not inv}.]
In the situation of Corollary \ref{cor DN not inv}, Theorem \ref{thm DN not inv} applies. And 
 by Lemmas \ref{lem P g tr cl} and \ref{lem eta DN eps}, 
\beq{eq ind DN non inv 3}
\lim_{\varepsilon \downarrow 0}
\ind_g(\hat D_{\varepsilon}) = \frac{1}{2} \lim_{\varepsilon \downarrow 0} \LIM_{t\downarrow 0}
(\Tr_g(e^{-t\tilde D_{\varepsilon, -} \tilde D_{\varepsilon, +}}) - \Tr_g(e^{-t\tilde D_{\varepsilon, +} \tilde D_{\varepsilon, -}})) - \frac{\Tr_g(P)+ \eta_g^{\reg}(D_N)}{2}.
\eeq
By Proposition \ref{prop proj lim small t}, $\eta_g(D_N)$ converges for small $t$ in the setting of Corollary \ref{cor DN not inv}, so $\eta_g^{\reg}(D_N) = \eta_g(D_N)$.

Now $\Tr_g(e^{-t\tilde D_{\varepsilon, \mp} \tilde D_{\varepsilon, \pm}})$ and the coefficients in its asymptotic expansion are continuous in $\varepsilon$. So the first term on the right hand side of \eqref{eq ind DN non inv 3} equals
\[
 \frac{1}{2}  \LIM_{t\downarrow 0}
(\Tr_g(e^{-t\tilde D_{-} \tilde D_{+}}) - \Tr_g(e^{-t\tilde D_{+} \tilde D_{-}})) = 
 \frac{1}{2}  \lim_{t\downarrow 0}
\bigl(\Tr_g(e^{-t\tilde D_{-} \tilde D_{+}}) - \Tr_g(e^{-t\tilde D_{+} \tilde D_{-}}) \bigr).
\]
And the right hand side equals the first term on the right hand side of \eqref{eq cor DN not inv} by the results cited in the proofs of Corollaries \ref{cor e}, \ref{cor APS discr} and \ref{cor APS ss} in Subsection \ref{sec pf APS}.
\end{proof}

\begin{remark} \label{rem eta reg direct}
In the proof of Corollary \ref{cor DN not inv}, we used Proposition \ref{prop proj lim small t} to find that $\eta_g^{\reg}(D_N) = \eta_g(D_N)$. It seems likely that this can also be deduced from the index formula, as in the proof of Theorem \ref{thm APS}. An argument of this type would involve analysing the expression in Lemma \ref{lem1} for the operator $\hat D_{\varepsilon}$ in the limit $\varepsilon \downarrow 0$. In situations where convergence of $\eta_g(D_N)$ for small $t$ is not clear, a version of Corollary \ref{cor DN not inv} holds with $\eta_g(D_N)$ replaced by $\eta_g^{\reg}(D_N)$. 
\end{remark}

\appendix

\section{The compact case}\label{sec cpt}

Suppose that $M$ is compact and $G$ is trivial. Then many arguments in this paper are unnecessary, and the others become simpler. We illustrate this by giving a sketch of a proof of the original {\APS} index theorem here. This may also help to clarify the overall idea of the proof of Theorem \ref{thm APS}. The details are simplified versions of the corresponding arguments in this paper.
We discuss the case where $D_N$ is invertible here. The general case then follows as in Section \ref{sec DN not inv}.

We use the same notation as in the rest of this paper.  In particular, consider the parametrices $\tilde Q$
of $\tilde D_+$, $Q_C'$
of $D_{C,+}$,
 and $R$
of $\hat D_+$ defined in \eqref{eq def tilde Q}, \eqref{eq def R R'} and \eqref{eq def QC'}, respectively.
%
In other approaches to the {\APS} index theorem, one often works with the heat operators $e^{-tD_- D_+}$ and $e^{-tD_+ D_-}$, or $e^{-t\hat D_- \hat D_+}$ and $e^{-t\hat D_+ \hat D_-}$, which are not trace class. (See for example Section 4 of the introduction to \cite{Melrose} for comments on this, and a solution in terms of $b$-calculus.) We work with the operators $1-R\hat D_+$ and $1-\hat D_+ R$ instead, which are $e$-trace class.

The {\APS} index of $D$ equals the index of $\hat D$ on $L^2$-sections, which by Lemma \ref{lem fin dim} equals
\begin{multline}\label{eq APS cpt}
\Tr_e(1-R\hat D_+) -\Tr_e(1-\hat D_+ R) = \Tr_e(\varphi_1(1-\tilde Q \tilde D_+)\psi_1) - \Tr_e(\varphi_1(1-\tilde D_+ \tilde Q)\psi_1)\\
+
\Tr_e ((\varphi_1 \tilde Q - \varphi_2 Q_C' )\sigma \psi_1' )
+ \Tr_e(\varphi_2 (D_{C,+}^{-1}-Q_C')\sigma \psi_2').
\end{multline}
The first two terms on the right hand side together equal $\frac{1}{2}\ind(\tilde D)$, the contribution from the interior of $M$. This is because the operators $\varphi_1(1-\tilde Q \tilde D_+)\psi_1$ and $\varphi_1(1-\tilde D_+ \tilde Q)\psi_1$ are trace class, so their $e$-traces equal their traces by Lemma \ref{lem trace class}.

 The third term on the right hand side of \eqref{eq APS cpt} equals
\[
-\Tr_e\Bigl(\int_0^t (\varphi_1 e^{-s\tilde D_-  \tilde D_+} - \varphi_2e^{-s D_{C,-} D_{C,+}}) D_{C,+}\sigma \psi_1' \, ds \Bigr).
\]
This goes to zero as $t\downarrow 0$, because the kernels of $e^{-s\tilde D_- \tilde D_+}$ and  $e^{-s D_{C,-} D_{C,+}}$ have the same asymptotic expansion on the support of $\psi_1'$, and $\varphi_1$ and $\varphi_2$ equal $1$ on  the support of $\psi_1'$.

The fourth term on the right hand side of \eqref{eq APS cpt} equals
\begin{multline*}
\Tr_e(\varphi_2 e^{-tD_{C,-}D_{C,+}} D_{C,+}^{-1} \sigma \psi_2' ) =
-\Tr_e \Bigl( \int_t^{\infty} \varphi_2 e^{-sD_{C,-} D_{C,+}} D_{C,-} \sigma \psi_2' \, ds \Bigr) \\
=
-\Tr_e \Bigl( \int_t^{\infty} \varphi_2 e^{-s \frac{\partial^2}{\partial u^2}} \frac{\partial}{\partial u} e^{-s D_N^2}\psi_2' \, ds \Bigr)
-\Tr_e \Bigl( \int_t^{\infty} \varphi_2 e^{-s \frac{\partial^2}{\partial u^2}} e^{-sD_N^2} D_N  \psi_2'\, ds  \Bigr).
\end{multline*}
Using the explicit form of the heat kernel on the real line, one shows that
the first term on the right hand side equals zero, while the second term equals
\beq{eq eta cpt}
-\int_t^{\infty} \frac{1}{\sqrt{4 \pi s}}
\Tr_e ( D_Ne^{-s D_N^2}) \, ds.
\eeq
The operator $D_Ne^{-s D_N^2}$ is trace class, so its $e$-trace equals its trace. 
Hence \eqref{eq eta cpt} tends to $-\frac{1}{2}\eta(D_N)$
as $t\downarrow 0$ (a priori we obtain the regularised $\eta$-invariant, but it then follows from the index formula that this is the actual $\eta$-invariant). The conclusion is that the {\APS-index formula holds:
\[
\ind_{\Aps}(D)  = \frac{1}{2} \ind(\tilde D)-\frac{1}{2}\eta(D_N).
\]
If $D$ is a $\Spinc$-Dirac operator twisted by a vector bundle $V$, then, by the Atiyah--Singer index theorem,  this becomes a special case of Corollary \ref{cor e}:
\[
\ind_{\Aps}(D) =
\int_M \hat A(M) e^{c_1(L)/2} e^{-R_V/2\pi i} -\frac{1}{2}\eta(D_N).
\]

\bibliographystyle{plain}

\bibliography{mybib}

\end{document}